\documentclass[12pt]{amsart}

\usepackage{amssymb,latexsym,amsmath,extarrows,mathrsfs,amsthm,color,amsrefs}
\usepackage{graphicx}

\usepackage[margin=1in, centering]{geometry}
\usepackage[bookmarksnumbered, colorlinks, plainpages]{hyperref}
\usepackage{tikz}
\usetikzlibrary{decorations.pathreplacing}
\hypersetup{
    colorlinks=true,
    linkcolor=blue,
    citecolor=magenta,
    filecolor=magenta,
    urlcolor=cyan
}

\usepackage{mathtools}
\mathtoolsset{showonlyrefs,showmanualtags}

\numberwithin{equation}{section}

\newtheorem{theorem}{Theorem}[section]

\newtheorem{lemma}[theorem]{Lemma}

\newtheorem{remark}[theorem]{Remark}

\newtheorem{corollary}[theorem]{Corollary}

\newcommand{\C}{\mathbb C}
\newcommand{\Z}{\mathbb Z}
\newcommand{\T}{\mathbb T}
\newcommand{\N}{\mathbb N}
\newcommand{\R}{\mathbb R}

\newcommand{\ol}{\overline}
\newcommand{\D}{\mathbb D}
\newcommand{\Q}{\mathbb Q}

\title[Arithmetic Spectral Transition]{Arithmetic spectral transition for the unitary almost Mathieu operator}
\author{Fan Yang}
\address{Department of Mathematics, Louisiana State University, Baton Rouge, LA 70803}
\thanks{Email: yangf@lsu.edu}

\begin{document}
\maketitle

\begin{abstract}
We study the unitary almost Mathieu operator (UAMO), a one-dimensional quasi-periodic unitary operator arising from a two-dimensional discrete-time quantum walk on $\mathbb Z^2$ in a homogeneous magnetic field. In the positive Lyapunov exponent regime $0\le \lambda_1<\lambda_2\le 1$, we establish an arithmetic localization statement governed by the frequency exponent $\beta(\omega)$. More precisely, for every irrational $\omega$ with $\beta(\omega)<L$, where $L>0$ denotes the Lyapunov exponent, and every non-resonant phase $\theta$, we prove Anderson localization, i.e.\ pure point spectrum with exponentially decaying eigenfunctions. This extends our previous arithmetic localization result for Diophantine frequencies (for which $\beta(\omega)=0$) to a sharp threshold in frequency.
\end{abstract}

\section{Introduction}

Discrete-time quantum walks were introduced by Aharonov--Davidovich--Zagury \cite{ADZ} as quantum analogues of classical random walks. In contrast to the diffusive behavior of the classical walk, quantum walks typically exhibit ballistic spreading, a feature underlying several quantum-algorithmic applications (search, element distinctness, verification, etc.), see e.g.\ \cite{Amb07,Por2013}. Beyond quantum information, quantum walks provide an experimentally accessible platform for studying wave propagation, interference, and localization phenomena in engineered lattices. Localization for random and disordered quantum walks (and, more generally, random unitary models) has been studied extensively; see, for example, \cite{ASW,HJS09,HJS,JM,J} and the spectral transition analysis for random quantum walks on trees in \cite{HJ}. For dynamical and spreading estimates for quantum walks via transfer-matrix bounds, see also \cite{DFO}.

A major mathematical development in this area is the connection between certain classes of quantum walks and CMV matrices \cite{CGMV}, which are unitary operators arising in the theory of orthogonal polynomials on the unit circle (OPUC) \cite{SOPUC1,SOPUC2,5years}. This bridge provides access to powerful spectral tools and has stimulated substantial activity on spectral and dynamical questions for quasi-periodic and random unitary models. For background on spectral theory of CMV matrices, including Weyl--Titchmarsh theory and dynamical characterizations via uniform hyperbolicity, see \cite{GZ,DFLY}. For small quasi-periodic Verblunsky coefficients, \cite{LDZ} proved purely absolutely continuous spectrum.

\subsection{The unitary almost Mathieu operator}\label{subsec:UAMO_intro}
We study a quasi-periodic quantum-walk operator introduced in \cite{CFO} as a generalization of the model from \cite{FOZ}, motivated by a two-dimensional quantum walk on $\mathbb Z^2$ in a homogeneous magnetic field (see Section~3 of \cite{CFO}). Quantum walks in external electric fields have also been studied; see, for example \cite{CRWAGW,Faninterval}. Let $\mathcal H=\ell^2(\mathbb Z)\otimes\mathbb C^2$ with standard basis $\varepsilon_n^s=\varepsilon_n\otimes e_s$, $s\in\{+,-\}$. Fix couplings $\lambda_1,\lambda_2\in[0,1]$ and write $\lambda_j'=\sqrt{1-\lambda_j^2}$, $j=1,2$. For frequency $\omega\in\T$ and phase $\theta\in\T$, define
\begin{equation}\label{eq:defW_intro_new}
W_{\lambda_1,\lambda_2,\omega,\theta}:=S_{\lambda_1}Q_{\lambda_2,\omega,\theta},
\end{equation}
where the shift $S_{\lambda_1}$ acts by
\[
S_{\lambda_1}\varepsilon_n^{\pm}=\lambda_1\varepsilon_{n\pm 1}^{\pm}\pm \lambda_1'\varepsilon_n^{\mp},
\]
and the quasi-periodic coin $Q_{\lambda_2,\omega,\theta}$ acts coordinate-wise via
\[
Q_{\lambda_2,\omega,\theta,n}=
\begin{pmatrix}
\lambda_2\cos(2\pi(\theta+n\omega))+i\lambda_2' & -\lambda_2\sin(2\pi(\theta+n\omega))\\
\lambda_2\sin(2\pi(\theta+n\omega)) & \lambda_2\cos(2\pi(\theta+n\omega))-i\lambda_2'
\end{pmatrix}.
\]
Following \cite{CFO}, we refer to \eqref{eq:defW_intro_new} as the \emph{unitary almost Mathieu operator} (UAMO), reflecting its close analogy with the self-adjoint almost Mathieu operator and the presence of sharp spectral transitions in the unitary setting.

It is convenient to place $W_{\lambda_1,\lambda_2,\omega,\theta}$ in the framework of (extended) CMV matrices. Under the identification
\[
\varepsilon_n^+ \mapsto \varepsilon_{2n},\qquad \varepsilon_n^- \mapsto \varepsilon_{2n+1},
\]
one may view $W_{\lambda_1,\lambda_2,\omega,\theta}$ as an extended CMV operator built from $2\times 2$ blocks $\Theta(\alpha_n)$ associated with the coefficients $\alpha_n$ defined in \eqref{def:alpha_n}; see Section~\ref{sec:CMV_background}. In particular, $W_{\lambda_1,\lambda_2,\omega,\theta}$ is unitary and $\sigma(W_{\lambda_1,\lambda_2,\omega,\theta})\subset\partial\mathbb D$.

A key structural feature of $W_{\lambda_1,\lambda_2,\omega,\theta}$ is that the Lyapunov exponent of the associated cocycle can be computed explicitly on the spectrum (Theorem~\ref{thm:LE}); this is an immediate application of Avila's global theory \cite{global} of one-frequency analytic cocycles in the UAMO setting, as implemented in \cite{CFO}. In particular, the supercritical (i.e.\ positive Lyapunov exponent) regime is precisely $0\le \lambda_1<\lambda_2\le 1$.

\subsection{Measure-theoretic vs.\ arithmetic localization}\label{subsec:background_intro}
In the positive Lyapunov exponent regime, Cedzich--Fillman--Ong \cite{CFO} established a detailed spectral picture for the unitary almost Mathieu operator and proved Anderson localization for a.e.\ $(\omega,\theta)$ (measure-theoretic in the parameters), adapting the Bourgain--Goldstein large deviation and semi-algebraic set scheme for analytic quasi-periodic Schr\"odinger operators \cite{BG} to the CMV setting. While robust, this approach does not yield localization for a prescribed irrational frequency.

In \cite{FanUAMO}, by adapting Jitomirskaya’s arithmetic localization strategy for the almost Mathieu operator \cite{jit99}, we proved Anderson localization for all Diophantine $\omega$ and for a full-measure set of phases $\theta$ described by an explicit non-resonance condition. The goal of the present paper is to extend this to a sharp arithmetic threshold in the frequency variable.

More broadly, sharp arithmetic localization and spectral transition phenomena for quasi-periodic operators have been studied intensively in recent years; see, for example, \cite{AYZ,JL1,JYinterface} for sharp transition results in the almost Mathieu setting and \cite{JLMaryland,JY,HJY} for related sharp analyses in the Maryland model, as well as other lattices \cite{Handry,ehmtransition,HYZ20,SSY}. 

The large deviation framework was advanced substantially in a series of works of Bourgain--Goldstein--Schlag \cite{GS01,GS08,GS11,GSV}. More recently, Han--Schlag introduced new ideas that elevate this analytic machinery to produce arithmetic localization statements for some general analytic potentials \cite{HS1,HS2,HS3}, including sharp results in Liouville regimes \cite{Hanloc}. These developments suggest that related large-deviation and resonance-elimination techniques should also be effective in the unitary setting, and may lead to further localization and spectral results beyond the UAMO.

\subsection{Arithmetic input and main result}\label{subsec:main_intro}
Throughout we work in the \emph{supercritical regime}
\begin{equation}\label{eq:supercritical_regime_intro}
0\le \lambda_1<\lambda_2\le 1.
\end{equation}
In this regime the Szeg\H{o} Lyapunov exponent on the spectrum is positive and explicit (see Sec. \ref{sec:Sze}): for all $z\in\sigma(W_{\lambda_1,\lambda_2,\omega,\theta})$,
\begin{equation}\label{eq:LE_equals_Lhalf}
L(\omega,S(\cdot,z))=\frac{L}{2},
\end{equation}
where we set
\begin{equation}\label{eq:def_L_intro}
L=L(\lambda_1,\lambda_2):=\ln\frac{\lambda_2(1+\lambda_1')}{\lambda_1(1+\lambda_2')}>0.
\end{equation}

Let $\omega\in\R\setminus\Q$ and write $p_n/q_n$ for its continued fraction approximants, see Sec.~\ref{sec:continued}. Define the frequency exponent
\begin{equation}\label{eq:def_beta_intro}
\beta(\omega):=\limsup_{n\to\infty}\frac{\ln q_{n+1}}{q_n}\in[0,\infty].
\end{equation}
This quantity measures the exponential quality of rational approximation to $\omega$; in particular, $\beta(\omega)=0$ for every Diophantine $\omega$.

As in the self-adjoint even-potential setting (cf.\ Jitomirskaya--Simon \cite{JS}), one should not expect Anderson localization to hold for all phases. We therefore restrict to a full-measure set of phases determined by the following non-resonance condition. For $\theta\in\T$, let $\|\theta\|=\mathrm{dist}(\theta,\Z)$ and set
\begin{equation}\label{eq:theta_set}
\tilde{\gamma}(\omega,\theta):=\limsup_{n\to\infty}
-\frac{\ln\|2\theta-\tfrac12+n\omega\|}{|n|}.
\end{equation}
We call $\theta$ \emph{non-resonant} (for $\omega$) if $\tilde{\gamma}(\omega,\theta)=0$.

Our main theorem establishes the localization side of the expected arithmetic transition: localization holds whenever the Lyapunov exponent dominates the frequency exponent.
\begin{theorem}[Arithmetic Anderson localization]\label{thm:main}
Assume \eqref{eq:supercritical_regime_intro}. Let $\omega\in\R\setminus\Q$ satisfy $\beta(\omega)<L$, and let $\theta\in\T$ be non-resonant, i.e.\ $\tilde{\gamma}(\omega,\theta)=0$.
Then the unitary almost Mathieu operator $W_{\lambda_1,\lambda_2,\omega,\theta}$
has Anderson localization; that is, it has pure point spectrum with exponentially decaying eigenfunctions.
\end{theorem}

\begin{remark}\label{rem:main}
\ 

\begin{enumerate}
\item Since $\beta(\omega)=0$ for every Diophantine $\omega$, Theorem~\ref{thm:main} recovers and extends \cite{FanUAMO}.
\item The condition $\beta(\omega)<L$ is the natural arithmetic threshold for the localization side of a sharp frequency transition, in analogy with arithmetic transitions for the self-adjoint almost Mathieu operator (see, e.g., \cite{AYZ,JL1}).
\end{enumerate}
\end{remark}

\subsection{Strategy and organization}\label{subsec:org_intro}
Our approach follows the arithmetic localization scheme: one seeks scale-by-scale Green's function control along the continued-fraction scales $q_n$, with resonance elimination governed by the competition between Lyapunov growth and small divisors. In the self-adjoint almost Mathieu setting, Jitomirskaya--Liu developed this scheme into a sharp arithmetic localization proof; see \cite{JL1}. In the CMV/unitary setting, finite-volume restrictions are block-structured and the relevant determinants enjoy additional algebraic constraints, which we exploit via the degree reduction established in \cite{FanUAMO}.
Compared to the Diophantine-frequency case studied in \cite{FanUAMO}, the main difficulties occur at strong Liouville scales. A key point is that the resonant analysis depends on a sharp understanding of eigenfunction decay in the surrounding non-resonant regions.

The paper is organized as follows. Section~2 collects the CMV/Szeg\H{o} background and arithmetic preliminaries. Section~3 recalls key determinant structure lemmas from \cite{FanUAMO}. Section~4 outlines the multi-scale analysis and the weak/strong Liouville dichotomy. Sections~5 and~6 treat eigenfunction decay in non-resonant and resonant regimes at strong Liouville scales, respectively. Section~7 sketches the weak Liouville scale argument.


\section{Preliminaries}\label{sec:prelim}

\subsection{Finite-volume restrictions}\label{sec:restriction}
Let $\mathbb D$ be the unit disk in $\C$ and $\partial\mathbb D$ the unit circle.
For a bounded operator (matrix) $A$ on $\ell^2(\Z)$ and integers $a\le b$, write
\[
A|_{[a,b]}:=\chi_{[a,b]}\,A\,\chi_{[a,b]}
\]
for its restriction to the interval $[a,b]$.
For extended CMV operators one has $(\mathcal L\mathcal M)|_{[a,b]}=\mathcal L|_{[a,b]}\,\mathcal M|_{[a,b]}$; see \cite{Kruger}.

\subsection{Extended CMV background}\label{sec:CMV_background}
For $\alpha\in\D$ set $\rho=\sqrt{1-|\alpha|^2}$ and define the CMV block
\[
\Theta(\alpha):=\begin{pmatrix}\ol{\alpha} & \rho\\ \rho & -\alpha\end{pmatrix}.
\]
Given a sequence $\{\alpha_n\}_{n\in\Z}\subset\D$, define $\Theta_n:=\Theta(\alpha_n)$ acting on
$\mathrm{span}\{\varepsilon_n,\varepsilon_{n+1}\}$. The (extended) CMV operator is
$W=\mathcal L\mathcal M$, where $\mathcal L$ is the direct sum of $\Theta_{2n}$ and $\mathcal M$ is the direct sum of $\Theta_{2n+1}$.
We refer to \cite{CGMV,SOPUC1,SOPUC2} for background and to \cite{CFO,FanUAMO} for the UAMO specialization.

\subsection{Szeg\H{o} cocycle and Lyapunov exponent}\label{sec:Sze}
We identify $\ell^2(\Z)\otimes \C^2 \to \ell^2(\Z)$ via
\[
\varepsilon_n^+ \mapsto \varepsilon_{2n},\qquad \varepsilon_n^- \mapsto \varepsilon_{2n+1}.
\]
For $n\in \Z$, set
\begin{align}\label{def:alpha_n}
\begin{cases}
\alpha_{2n}:=\lambda_1', \qquad \rho_{2n}:=\lambda_1,\\
\alpha_{2n+1}:=\lambda_2\sin(2\pi(\theta+n\omega)),\qquad
\rho_{2n+1}:=\lambda_2\cos(2\pi(\theta+n\omega))+i\lambda_2'.
\end{cases}
\end{align}

For a sequence $\{\alpha_n\}\subset \D$, define the Szeg\H{o} transfer matrix by
\begin{equation}\label{eq:def_Szego}
S_{n,z}:=\frac{1}{|\rho_n|}
\begin{pmatrix}
z &-\ol{\alpha}_n\\
-\alpha_n z &1
\end{pmatrix}.
\end{equation}
The Lyapunov exponent of the Szeg\H{o} cocycle is
\begin{align}\label{def:LS}
L(\omega, S(\cdot,z)):=\lim_{n\to\infty} \frac{1}{n} \int_{\T} \ln \|S_n(\theta,z)\|\, d\theta,
\end{align}
where $S_n(\theta,z):=S_{n-1,z}(\theta+(n-1)\omega)\cdots S_{0,z}(\theta)$.
This limit exists by subadditivity.

\begin{theorem}[Corollary 2.8 of \cite{CFO} and (3.25) of \cite{FanUAMO}]\label{thm:LE}
For $z\in \sigma(W_{\lambda_1,\lambda_2,\omega,\theta})$, we have
\[
L(\omega, S(\cdot, z))=\max\left(0, \frac{1}{2}\ln \frac{\lambda_2(1+\lambda_1')}{\lambda_1(1+\lambda_2')}\right).
\]
\end{theorem}
In particular, in the regime \eqref{eq:supercritical_regime_intro} we have \eqref{eq:LE_equals_Lhalf}.

The following integral was computed in \cite{CFO}:
\begin{align}\label{eq:int_rho}
\int_{\T}\ln |\rho_{2n}\rho_{2n+1}|\, d\theta
=\int_{\T}\ln |\lambda_1(\lambda_2\cos(2\pi\theta)+i\lambda_2')|\, d\theta
=\ln\frac{\lambda_1(1+\lambda_2')}{2}.
\end{align}

Let
\begin{equation}\label{eq:def_Lpm}
L_+:=\ln\frac{\lambda_2(1+\lambda_1')}{2},
\qquad
L_-:=\ln\frac{\lambda_1(1+\lambda_2')}{2},
\qquad
L:=L_+-L_-.
\end{equation}

\begin{lemma}\label{lem:upperbounds}$\mathrm{(}$e.g.\ \cite{Furman, JMavi}$\mathrm{)}$
Let $(\omega, H)$ be a continuous cocycle. Then for any $\varepsilon>0$, for $|k|$ large enough,
\[
\|H_k(\theta)\|\leq e^{|k|(L(\omega, H)+\varepsilon)}\qquad \text{for all }\theta\in\T,
\]
where $H_k(\theta):=H(\theta+(k-1)\omega)\cdots H(\theta)$.
\end{lemma}

\begin{corollary}\label{cor:1cocycle}
If $g$ is continuous and $\ln |g|\in L^1(\T)$, then for any $\varepsilon > 0$ and $b-a$ sufficiently large,
\begin{equation}\label{eq:scalar_upper}
\left|\prod_{j=a}^b g (\theta+j\omega)\right| \leq
e^{(b-a+1)\left(\int_{\T} \ln|g|\, d\theta+\varepsilon\right)}.
\end{equation}
\end{corollary}

As a direct application of Corollary \ref{cor:1cocycle}, we have
\begin{equation}\label{eq:rho_bd}
\left|\prod_{j=a}^b\rho_{j}\right| \leq e^{-\frac{|b-a|}{2}(L_- +\varepsilon)}.
\end{equation}

\subsection{Continued fractions and arithmetic exponents}\label{sec:continued}
Let $\omega\in \T\setminus \Q$. Then $\omega$ has the unique continued fraction expansion
\[
\omega=\frac{1}{a_1+\frac{1}{a_2+\frac{1}{a_3+\cdots}}},\qquad a_n\in\N.
\]
Let
\begin{equation}\label{defpnqn}
\frac{p_n}{q_n}=\frac{1}{a_1+\frac{1}{a_2+\frac{1}{\cdots+\frac{1}{a_n}}}}
\end{equation}
be the continued fraction approximants of $\omega$.
We recall that
\begin{equation}\label{appro1}
\|k\omega\|_{\T} \geq \|q_{n-1}\omega\|_{\T}\quad \text{for any }0\neq |k|<q_n,
\end{equation}
and
\begin{equation}\label{appro2}
\frac{1}{2q_{n+1}} \leq \|q_n \omega \|_{\T} \leq \frac{1}{q_{n+1}}.
\end{equation}

\subsection{Poisson formula and Szeg\H{o} connection}\label{sec:Poisson}
Let $P_{[a,b],z}:=\det(z-W|_{[a,b]})$ be the Dirichlet determinant and let $\Psi$ be a generalized solution to
$W_{\lambda_1,\lambda_2,\omega,\theta}\Psi=z\Psi$. The Poisson formula below is used repeatedly in this paper:
\begin{align}\label{eq:Poisson_1}
|\Psi_y|\leq C_{\lambda_1,\lambda_2}
&\left(\prod_{j=a}^{y-1} |\rho_j|\right) \frac{|P_{[y+1,b],z}|}{|P_{[a,b],z}|}\max(|\Psi_{a-1}|, |\Psi_a|)\\
&+C_{\lambda_1,\lambda_2}\left(\prod_{j=y}^{b-1} |\rho_j|\right) \frac{|P_{[a,y-1],z}|}{|P_{[a,b],z}|}\max(|\Psi_b|, |\Psi_{b+1}|). \notag
\end{align}

We also use the standard reversed (``star'') operation on polynomials: for
$p(z)=\sum_{j=0}^d p_j z^j$,
\begin{equation}\label{eq:def_star}
p^*(z):=z^d\,\overline{p(1/\overline z)}.
\end{equation}
(Here $p^*$ denotes the reversed polynomial, not the adjoint of a matrix.)
Moreover, we use the following connection between $P_{[a,b],z}$ and the Szeg\H{o} cocycle (see \cite{Wang}):
\begin{align}\label{eq:Sze2}
S_{b,z}\cdots S_{a,z}=\frac{1}{\prod_{j=a}^b |\rho_j|}
\begin{pmatrix}
zP_{[a+1,b],z} & \dfrac{zP_{[a+1,b],z}-P_{[a,b],z}}{\alpha_{a-1}}\\[1mm]
z\left(\dfrac{zP_{[a+1,b],z}-P_{[a,b],z}}{\alpha_{a-1}}\right)^{*}
& (P_{[a+1,b],z})^{*}
\end{pmatrix}.
\end{align}

\section{Preliminary lemmas}\label{sec:key}

\begin{lemma}[Lagrange interpolation]\label{lem:Lagrange}
Let $g_n(\xi)$ be a polynomial in $\xi$ of degree at most $n$.
For any $\theta_1,\theta_2,\dots,\theta_{n+1}\in \T$, we have
\begin{equation}\label{eq:Lagrange}
g_n(\xi)= \sum_{j=1}^{n+1} g_n(\sin(2\pi\theta_j))
\prod_{\substack{\ell=1\\ \ell \neq j}}^{n+1}
\frac{\xi-\sin(2\pi \theta_{\ell})}{\sin(2\pi \theta_j)-\sin(2\pi \theta_{\ell})}.
\end{equation}
\end{lemma}

The following lemma from \cite{AJ1} gives a useful estimate for products of cosines appearing in our analysis.
\begin{lemma}\label{lem:nonmin}
Let $\omega\in \R\setminus \Q $, $\theta\in\R$ and $0\leq j_0 \leq q_{n}-1$ be such that
\[
|\cos \pi(\theta+j_{0}\omega)|=\inf_{0\leq j \leq q_{n}-1} |\cos \pi(\theta+j\omega)|.
\]
Then for some absolute constant $C$,
\[
-C\ln q_{n} \leq \sum_{j=0,\, j\neq j_0}^{q_{n}-1} \ln |\cos \pi (\theta+j\omega)|+(q_{n}-1)\ln2 \leq C\ln q_n.
\]
\end{lemma}

\begin{lemma}\label{lem:main}\cite[Lemma 4.1]{FanUAMO}
There exists a polynomial $g_n(\xi)$ in $\xi$ of degree at most $n$ such that
\[
P_{[1,2n],z}(\theta)=g_n\!\left(\sin\bigl(2\pi (\theta+\tfrac{n-1}{2}\omega)\bigr)\right).
\]
\end{lemma}

This implies that $P_{[1,2n],z}$ cannot be uniformly small at $n+1$ well-distributed choices of $\theta$.

\begin{lemma}\label{lem:ave_low}\cite[Lemma 4.2]{FanUAMO}
For any $\varepsilon>0$, for $n>N(\varepsilon)$ large enough,
\[
\frac{1}{2n}\int_{\T} \ln|P_{[1,2n],z}(\theta)|\, d\theta\geq
\frac{1}{2}\ln\!\left(\frac{\lambda_2(1+\lambda_1')}{2}\right)-\varepsilon
\geq \frac{L_+}{2}-\varepsilon.
\]
\end{lemma}

\section{An outline of the proof}\label{sec:outline}
By Shnol's theorem (see, e.g., \cite{Shnol57,Simon81,HanSchnol}), in order to prove Anderson localization, it suffices to show that every generalized eigenfunction $\Psi$ solving
$W_{\lambda_1,\lambda_2,\omega,\theta}\Psi=z\Psi$, normalized by $\|(\Psi_0,\Psi_{-1})\|=1$ and satisfying a polynomial bound $|\Psi_y|\leq C(1+|y|)$ for some $C>0$, decays exponentially.

Fix $0<\varepsilon_0<\frac{1}{100}(L-\beta(\omega))$.
Choose $\varepsilon>0$ such that
\begin{equation}\label{eq:choose_eps}
(4\beta(\omega)+250)\varepsilon<\varepsilon_0.
\end{equation}
We prove exponential decay of $\Psi_y$ at scale $n$, i.e.\ for $q_n^{1-\varepsilon}<|y|<4q_{n+1}^{1-\varepsilon}$, for all large enough $n$.
Without loss of generality, we assume $y>0$ throughout.

We divide scales into two cases.

\underline{Strong Liouville scale: $q_{n+1}>e^{\varepsilon q_n}$.}
The proof of exponential decay at such a scale is the main challenge.
We divide the interval $[q_n^{1-\varepsilon}, 4q_{n+1}^{1-\varepsilon}]$ into resonant regimes (sites near multiples of $2q_n$) and non-resonant regimes (sites between two consecutive such multiples). The non-resonant and resonant analyses are presented in Sections~\ref{sec:non_res} and \ref{sec:res}, respectively.

\smallskip\noindent

\underline{Weak Liouville scale: $q_{n+1}\leq e^{\varepsilon q_n}$.}
The proof of exponential decay at such a scale is similar to the Diophantine-frequency analysis in \cite{FanUAMO}. We briefly discuss this case in Section~\ref{sec:Diophantine}, and only focus on the additional technical difficulties compared to the Diophantine case.

\smallskip\noindent

\section{Eigenfunction in the non-resonant regimes in the strong Liouville case}\label{sec:non_res}
We consider $q_n^{1-\varepsilon}<y<10q_{n+1}^{1-\varepsilon}$ with $y$ sufficiently large (hence $n$ large).
We call $y$ \emph{resonant} at scale $q_n$ if $\mathrm{dist}(y,2\Z q_n)\leq q_n^{1-\varepsilon}$; otherwise we call $y$ \emph{non-resonant}.

For $t\in \Z$, denote the resonance interval by
\[
R_{2t}:=\bigl[\,2tq_n-q_n^{1-\varepsilon},\,2tq_n+q_n^{1-\varepsilon}\,\bigr].
\]
Let $r_{2t}:=\max_{y\in R_{2t}} |\Psi_y|$.
Define the non-resonance interval
\[
\mathcal{N}_{2(t-1),2t}:=\bigl[(2t-2)q_n+q_n^{1-\varepsilon},\,2tq_n-q_n^{1-\varepsilon}\bigr].
\]

The goal of this section is to prove the following:
\begin{theorem}\label{lem:non-res}
For any $y\in \mathcal{N}_{2(t-1),2t}$,
\[
|\Psi_{y}|\leq
e^{-\left(\frac{L}{2}-12\varepsilon_0\right)\,(\mathrm{dist}(y,2\Z q_n)-q_n^{1-\varepsilon})}\,
\max(r_{2(t-1)}, r_{2t}).
\]
\end{theorem}

Throughout the proof, the following choice of $m$ and $s$ are crucial. For 
$y\in \mathcal{N}_{2(t-1),2t}$, let $m$  be the large integer such that $\mathrm{dist}(y,2\Z q_n) \geq 8q_m$, then choose $s$ to be the largest integer such that $8sq_m \leq \mathrm{dist}(y,2\Z q_n)< 8(s+1)q_m $. Due to such choices, clearly one has
\begin{small}
    \begin{align}\label{eq:s_qm}
\max(q_n^{1-\varepsilon}, 8sq_m)\leq \mathrm{dist}(y,2\Z q_n)<\min(8(s+1)q_m, 8q_{m+1})\leq \min(16s q_m, 8q_{m+1}).
\end{align}
\end{small}

Let $h=8sq_m-2$ and set 
\begin{align}\label{def:I0Iy}
\begin{cases}
I_0:=[-7sq_m,-5sq_m+1]\cap (2\Z+1),\\
I_y:=[y-7sq_m+2,y-sq_m+1]\cap (2\Z+1).
\end{cases}
\end{align}
Note that $I_0\cap I_y=\emptyset$ and $|I_0|+|I_y|=4sq_m$.

As an overview of the section, we will first prove in Lemma \ref{lem:I1_I2_large_nonres} that there exists (at least) one $x_1\in I_0\cup I_y$ such that $|P_{[x_1,x_1+h-1]}(\theta)|$ has the expected lower bound. We will then prove by contradiction that such an $x_1$ can not be in $I_0$, since that would lead to exponential decay of $|\Psi_0|$ and $|\Psi_{-1}|$. Hence such $x_1$ must be in $I_y$, see Lemma \ref{lem:I2_large_nonres}. The existence of $x_1\in I_y$ will be fed into the Poisson formula to yield exponential decay of $|\Psi_y|$, which leads to the proof of Theorem \ref{lem:non-res}.

\begin{lemma}\label{lem:I1_I2_large_nonres}
For $y$ large enough,
\begin{align}
    \max_{x_1\in I_0\cup I_y} |P_{[x_1,x_1+h-1]}(\theta)|\geq e^{\frac{h}{2}({L_+}-2\varepsilon_0)}.
\end{align}
\end{lemma}
\begin{proof}[Proof of Lemma \ref{lem:I1_I2_large_nonres}]
    Proof by contradiction, assume that 
    \begin{align}\label{eq:assume_P_small}
        \max_{x_1\in I_0\cup I_y}|P_{[x_1,x_1+h-1]}(\theta)|<e^{\frac{h}{2}({L_+}-2\varepsilon_0)}.
    \end{align}    
Denote $\theta_{\ell} := \theta+\frac{\ell-1}{2}\omega+\frac{h/2-1}{2}\omega-\frac{1}{4}$, we have $\sin(2\pi (\theta+\frac{\ell-1}{2}\omega+\frac{h/2-1}{2}\omega))=\cos(2\pi\theta_{\ell})$.
Note that by Lemma \ref{lem:main} and the Lagrange interpolation, we have 
\begin{align}\label{eq:gh/2_Lag_1}
g_{h/2}(\xi)
=&\sum_{x_1\in I_0\cup I_y} g_{h/2}(\cos(2\pi \theta_{x_1})) \cdot \prod_{\substack{\ell\in I_0\cup I_y\\ \ell\neq x_1}} \frac{\xi-\cos(2\pi\theta_{\ell})}{\cos(2\pi\theta_{x_1})-\cos(2\pi\theta_{\ell})}\\
=&\sum_{x_1\in I_0\cup I_y} P_{[x_1,x_1+h-1]}(\theta)
 \cdot  \prod_{\substack{\ell\in I_0\cup I_y\\ \ell\neq x_1}} \frac{\xi-\cos(2\pi \theta_{\ell})}{\cos(2\pi \theta_{x_1})-\cos(2\pi \theta_{\ell})}.
\end{align}
Let $\xi=\cos(2\pi\tilde{\theta})$ and denote
\begin{align}
 U(\xi,x_1,h, I_0\cup I_y):=\prod_{\substack{\ell\in I_0\cup I_y\\ \ell\neq x_1}} \frac{\cos(2\pi \tilde{\theta})-\cos(2\pi \theta_{\ell})}{\cos(2\pi \theta_{x_1})-\cos(2\pi \theta_{\ell})}.
\end{align}
Hence by \eqref{eq:assume_P_small} and \eqref{eq:gh/2_Lag_1},
\begin{align}\label{eq:P<U_non_res}
\sup_{\theta\in \T} |P_{[1,h]}(\tilde{\theta})|
=&\sup_{\xi\in [-1,1]}|g_{h/2}(\xi)|\notag\\
\leq &(h/2+1)\max_{x_1\in I_0\cup I_y}|P_{[x_1,x_1+h-1]}(\theta)|\cdot \sup_{\xi\in [-1,1]}\max_{x_1\in I_0\cup I_y} |U(\xi,x_1,h,I_0\cup I_y)|\notag\\
\leq &(h/2+1) e^{\frac{h}{2}(L_+-2\varepsilon_0)}\cdot \sup_{\xi\in [-1,1]}\max_{x_1\in I_0\cup I_y} |U(\xi,x_1,h,I_0\cup I_y)|.
\end{align}
Combine this with Lemma \ref{lem:non_res_uni} below, we get 
\begin{align}
    \sup_{\theta\in \T}|P_{[1,h]}(\tilde{\theta})|\leq e^{\frac{h}{2}(L_+-\varepsilon_0)},
\end{align}
which is a contradiction to Lemma \ref{lem:ave_low}.
\end{proof}
Hence in order to prove Lemma \ref{lem:I1_I2_large_nonres}, it suffices to prove the following lemma.
\begin{lemma}\label{lem:non_res_uni}
We have
\begin{align}
        \sup_{\tilde{\theta}\in \T}\max_{x_1\in I_0\cup I_y} |U(\cos(2\pi\tilde{\theta}),x_1,h,I_0\cup I_y)|\leq e^{\frac{\varepsilon_0 h}{2}}.
\end{align}
\end{lemma}

\begin{proof}[Proof of Lemma \ref{lem:non_res_uni}]
Without loss of generality we assume $y$ is even and $7sq_m$ is even. The other cases are small variants of this case. 
We have
\begin{align}\label{eq:U_sum1sum2}
    &\ln |U(\cos(2\pi\tilde{\theta}),x_1, h,I_0\cup I_y)|\notag\\
    =&\sum_{\substack{\ell\in I_0\cup I_y\\ \ell\neq x_1}}\ln |\cos(2\pi\tilde{\theta})-\cos (2\pi\theta_{\ell})|
-\sum_{\substack{\ell\in I_0\cup I_y\\ \ell\neq x_1}} \ln |(\cos (2\pi\theta_{x_1})-\cos (2\pi\theta_{\ell})|,\notag\\
=&:\sum_1-\sum_2,
\end{align}

By Lemma \ref{lem:nonmin}, 
\begin{align}\label{eq:sum1_nonres}
\sum_1=&\sum_{\substack{\ell\in I_0\cup I_y\\ \ell\neq x_1}} \ln |2\sin(\pi(\tilde{\theta}-\theta_{\ell})
\sin(\pi(\tilde{\theta}+\theta_{\ell}))|\notag\\
= &(4sq_m-1)\ln 2+\sum_{\substack{\ell\in I_0\cup I_y\\ \ell\neq x_1}}\ln |\sin(\pi(\tilde{\theta}-\theta_{\ell}))|+\sum_{\substack{\ell\in I_0\cup I_y\\ \ell\neq x_1}} \ln |\sin(\pi(\tilde{\theta}+\theta_{\ell}))| \notag\\
\leq &C s \ln q_m-4sq_m \ln2 \notag\\
\leq &4sq_m(-\ln 2+\varepsilon),
\end{align}
where in the last line we used 
$$\frac{\ln q_m}{q_m}<\frac{\varepsilon}{4C},$$
which holds for large enough $n$. Note that when $n$ is large enough, $m$ is also large, since $q_{m+1}>\mathrm{dist}(y,2q_n\Z)/8>q_n^{1-\varepsilon}/8$ by \eqref{eq:s_qm}.
We shall use this fact multiple times later without explicitly mentioning it.

Combining \eqref{eq:sum1_nonres} with Lemma \ref{lem:sum2_nonres} below, we will arrive at
\begin{align}
\sum_1-\sum_2 &\leq 4sq_m(-\ln2+\varepsilon)-4sq_m\ln 2-\sum_{2,1}-\sum_{2,2}\\
&\leq -8sq_m\ln 2+4\varepsilon sq_m-(-8sq_m\ln 2-(16\beta+35)\varepsilon sq_m)\\
&\leq (16\beta+40)\varepsilon sq_m<\frac{\varepsilon_0 h}{2}.
\end{align}
This proves Lemma \ref{lem:non_res_uni}.
\end{proof}

So next, we will prove the lower bound of $\sum_2$.

\subsection{Estimates of $\sum_2$}\label{sec:sum2_nonres}

\begin{lemma}\label{lem:sum2_nonres}
We have
    \begin{align}
        \sum_2\geq -8sq_m\ln 2-(16\beta+35)\varepsilon sq_m.
    \end{align}
\end{lemma}
Proof of Lemma \ref{lem:sum2_nonres}: Clearly,
\begin{align}\label{eq:sum2_nonres}
\sum_2
=&(4sq_m-1)\ln 2+\sum_{\substack{\ell\in I_0\cup I_y\\ \ell\neq x_1}}\ln|\sin(\pi\frac{x_1-\ell}{2}\omega)|\notag\\
&\qquad +
\sum_{\substack{\ell\in I_0\cup I_y\\ \ell\neq x_1}}\ln|\sin(\pi(2\theta-\frac{1}{2}+(\frac{x_1+\ell}{2}+4sq_m-3)\omega))| \notag\\
=:&(4sq_m-1)\ln 2+\sum_{2,1}+\sum_{2,2}.
\end{align}
We further break $I_0, I_y$ into $I_y=:\cup_{j=1}^{3s} T_j$ and $I_0=:\cup_{j=3s+1}^{4s} T_j$,
where
\begin{align}\label{def:Tj}
T_j=\begin{cases}
[y-7sq_m+2(j-1)q_m, y-7sq_m+2j q_m-1]\cap (2\Z+1), \text{ for } 1\leq j\leq 3s\\
[-7sq_m+2(j-3s-1)q_m, -7sq_m+2(j-3s)q_m-1]\cap (2\Z+1), \text{ for } 3s+1\leq j\leq 4s.
\end{cases}
\end{align}
For $w=1,...,4$, let
\begin{align}\label{def:T^w}
    T^{(w)}:=\cup_{j=(w-1)s}^{ws-1}T_j.
\end{align}

Clearly, the lower bound of $\sum_2$ follows from those for $\sum_{2,1}$ and $\sum_{2,2}$, which we address in Sections \ref{sec:sum21_nonres} and \ref{sec:sum22_nonres} respectively below.

\subsubsection{Estimates of $\sum_{2,1}$}\label{sec:sum21_nonres}
\begin{lemma}\label{lem:sum21}
We have
    \begin{align}
        \sum_{2,1}\geq -4sq_m\ln 2-(8\beta+17)\varepsilon sq_m.
    \end{align}
\end{lemma}
\begin{proof}[Proof of Lemma \ref{lem:sum21}]
Without loss of generality, we assume $x_1\in T_{j_0}$ for some $j_0\leq 3s$. Note this implies $x_1\in I_y$. 
By Lemma \ref{lem:nonmin},
\begin{align}\label{eq:sum21_nonres}
\sum_{2,1}&=\sum_{\substack{\ell\in I_0\cup I_y\\ \ell\neq x_1}}\ln|\sin(\pi\frac{x_1-\ell}{2}\omega)|\notag\\
&\geq -4sq_m\ln 2-Cs\ln q_m+\sum_{\substack{j=1\\ j\neq j_0}}^{4s} \min_{\ell\in T_j}\ln\|\frac{x_1-\ell}{2}\omega\|,
\end{align}
above we used, $|\sin\pi\theta|\geq 2\|\theta\|$. The estimate of $\sum_{2,1}$ is reduced to the following estimate.
\begin{lemma}\label{lem:allmin_21}
    \begin{align}
        \sum_{\substack{j=1\\ j\neq j_0}}^{4s}\, \min_{\ell\in T_j} \ln \|\frac{x_1-\ell}{2}\omega\|\geq -(8\beta+16)\varepsilon sq_m.
    \end{align}
\end{lemma}
Clearly, combining Lemma \ref{lem:allmin_21} with \eqref{eq:sum21_nonres} yields the claimed result of Lemma \ref{lem:sum21}. 
\end{proof}

Next, we prove Lemma \ref{lem:allmin_21}.
\begin{proof}[Proof of Lemma \ref{lem:allmin_21}]
In order to estimate the sum of the $(4s-1)$ minimums $\sum_{\substack{j=1\\ j\neq j_0}}^{4s}$, we first control the minimum of the $(4s-1)$ minimums in Lemma \ref{lem:minmin_21} below.
Let 
 \[\min_{\substack{r=1\\ r\neq j_0}}^{4s} \min_{\ell\in T_r}\ln\|\frac{x_1-\ell}{2}\omega\| =: \ln\|\frac{x_1-\hat{\ell}}{2}\omega\|.\]
\begin{lemma}\label{lem:minmin_21}
The minimum of the $(4s-1)$ minimums satisfies the following lower bound:
    \begin{align}
        \ln \|\frac{x_1-\hat{\ell}}{2}\omega\|\geq -\varepsilon sq_m.
    \end{align}
\end{lemma}
\begin{proof}[Proof of Lemma \ref{lem:minmin_21}]
We discuss two cases depending on the location of $\hat{\ell}$.

\underline{Case 1. $\hat{\ell} \in  I_y$}. 
In this case, by \eqref{eq:s_qm}, we have 
\begin{align}
0<\frac{|x_1-\hat{\ell}|}{2}\leq 3sq_m<\frac{3}{8}\mathrm{dist}(y,2\Z q_n)\leq \frac{3}{8}q_n.
\end{align}
Hence
\begin{align}\label{eq:minmin_21}
    \ln \|\frac{x_1-\hat{\ell}}{2}\omega\|\geq \ln \|q_{n-1}\omega\| \geq -\ln (2q_n).
\end{align}
Again, by \eqref{eq:s_qm}, we have
\begin{align}\label{eq:epsqn<16sqm}
    q_n^{1-\varepsilon}\leq 16sq_m.
\end{align}
Combining this into \eqref{eq:minmin_21}, we have
\begin{align}\label{eq:minmin_21_final}
    \ln\|\frac{x_1-\hat{\ell}}{2}\omega\|\geq -\varepsilon sq_m.
\end{align}

\underline{Case 2. $\hat{\ell} \in I_0$}.

Let $\hat{\ell}+2tq_n=:\tilde{\ell}$.
Note $\tilde{\ell}\geq 2tq_n-7sq_m$, while $x_1\leq y-sq_m\leq 2tq_m-9sq_m<\tilde{\ell}$. In particular $x_1\neq \tilde{\ell}$.

Since both $x_1$ and $\tilde{\ell}$ lie between $2(t-1)q_n$ and $2tq_n$, we have
    $0<|(x_1-\tilde{\ell})/2|<q_n$.
This implies
\begin{align}
    \|\frac{x_1-\tilde{\ell}}{2}\omega\|\geq \|q_{n-1}\omega\|\geq \frac{1}{2q_n}.
\end{align}
Also note that since $tq_n< 10q_{n+1}^{1-\varepsilon}$ and $q_{n+1}\geq e^{\varepsilon q_n}$, we have
\begin{align}
    \|tq_n\omega\|=t\|q_n\omega\|\leq \frac{t}{q_{n+1}}\leq 10q_{n+1}^{-\varepsilon}\leq 10e^{-\varepsilon^2 q_n}\ll \frac{1}{2q_n}.
\end{align}
Hence
\begin{align}
    \|\frac{x_1-\hat{\ell}}{2}\omega\|\geq \|\frac{x_1-\tilde{\ell}}{2}\omega\|-\|tq_n\omega\|\geq \frac{1}{3q_n}.
\end{align}
Therefore
\begin{align}
   \ln\|\frac{x_1-\hat{\ell}}{2}\omega\|\geq \ln \frac{1}{3q_n}\geq -\varepsilon sq_m,
\end{align}
where we used the same arguments as in Case 1 in the last estimate. This completes the proof of Lemma \ref{lem:minmin_21}.
\end{proof}

Next, we continue the proof of Lemma \ref{lem:allmin_21} as follows.
For each $w=1,...,4$, we define
\begin{align}
    S_w:=\sum_{\substack{j=(w-1)s+1\\ j\neq j_0}}^{ws} \min_{\ell\in T_j} \ln \|\frac{x_1-\ell}{2}\omega\|.
\end{align}
Clearly,
\begin{align}\label{eq:summin=Sw}
    \sum_{\substack{j=1\\ j\neq j_0}}^{4s}\, \min_{\ell\in T_j} \ln \|\frac{x_1-\ell}{2}\omega\|
    =&\sum_{w=1}^4 S_w.
\end{align}
Within each sum $S_w$, the $\ell$'s are chosen from the same $T^{(w)}$, see \eqref{def:T^w}, hence for arbitrary two different $\ell, \ell'\in T^{(w)}$, we have $|\ell-\ell'|\leq 2sq_m<2q_{m+1}$, see \eqref{eq:s_qm}.
Therefore
\begin{align}\label{eq:x1-ell-x1-ell'}
    \|\frac{x_1-\ell}{2}\omega-\frac{x_1-\ell'}{2}\omega\|\geq \|q_m\omega\|.
\end{align}
This way, we can estimate
\begin{align}\label{eq:Sw_0}
    S_w
    \geq &2\min_{\substack{j=1\\ j\neq j_0}}^{4s} \min_{\ell\in T_j} \ln \|\frac{x_1-\ell}{2}\omega\|+2\sum_{j=1}^s \ln (j\|q_m\omega\|)\\
    \geq &-2\varepsilon sq_m+2s\ln \frac{s}{q_{m+1}},
\end{align}
where we used Lemma \ref{lem:minmin_21} for the minimum of minimums.
Now we continue to lower bound $s\ln (s/q_{m+1})$ as follows. Note that $8q_m\leq \mathrm{dist}(y,2\Z q_n)\leq q_n$, hence $q_{m+1}\leq q_n$. Also recall $16sq_m\geq q_n^{1-\varepsilon}$, see \eqref{eq:s_qm}. Therefore
\begin{align}
    s\ln \frac{s}{q_{m+1}}\geq s\ln \frac{q_n^{1-\varepsilon}}{16q_mq_{m+1}}
    \geq &s\ln \frac{1}{16q_mq_{m+1}^{\varepsilon}}\\
    \geq &-s\ln (16q_m)-\varepsilon sq_m\frac{\ln q_{m+1}}{q_m}\\
    \geq &-(\beta+1)\varepsilon sq_m,
\end{align}
where we used that $(\ln q_{m+1})/q_m<\beta+\varepsilon$.
Plugging this into \eqref{eq:Sw_0}, we have
\begin{align}\label{eq:Sw}
S_w\geq -(2\beta+4) \varepsilon sq_m.
\end{align}
This completes the proof of Lemma \ref{lem:allmin_21} due to \eqref{eq:summin=Sw}, and hence the analysis of $\sum_{2,1}$.
\end{proof}

Next, we move on to estimate $\sum_{2,2}$.
\subsubsection{Estimates of $\sum_{2,2}$}\label{sec:sum22_nonres}
\begin{lemma}\label{lem:sum22}
    \begin{align}
        \sum_{2,2}\geq -4sq_m\ln 2-(8\beta+17)\varepsilon sq_m.
    \end{align}
\end{lemma}
\begin{proof}[Proof of Lemma \ref{lem:sum22}]
Similar to \eqref{eq:sum21_nonres}, we have
\begin{align}\label{eq:sum22}
\sum_{2,2}=&\sum_{\substack{\ell\in I_0\cup I_y\\ \ell\neq x_1}}\ln|\sin(\pi(2\theta-1/2+(\frac{x_1+\ell}{2}+4sq_m-3)\omega))|\notag\\
\geq &-4sq_m\ln 2-\varepsilon sq_m+\sum_{j=1}^{4s} \min_{\ell\in T_j}\ln\|2\theta-1/2+(\frac{x_1+\ell}{2}+4sq_m-3)\omega\|
\end{align}
We again estimate the minimum of the $4s$ minimums first.
\begin{lemma}\label{lem:minmin22}
    \begin{align}
        \min_{j=1}^{4s}\min_{\ell\in T_j}\ln\|2\theta-1/2+(\frac{x_1+\ell}{2}+4sq_m-3)\omega\|\geq -21\varepsilon^3 sq_m.
    \end{align}
\end{lemma}
\begin{proof}[Proof of Lemma \ref{lem:minmin22}]
We divide into three different cases and we will use $\tilde{\gamma}(\omega, \theta)=0$, so for any sufficient large $n$, 
\begin{align}\label{eq:gamma=0_eps}
  \|2\theta-1/2+n\omega\|\geq  e^{-\varepsilon^3 |n|}.
\end{align}
Within the proof, we divide into three cases, depending on the locations of $x_1$ and $\ell$.

\underline{Case 1. Both $x_1, \ell \in I_y$}.

Without loss of generality, we assume $y\geq 2(t-1)q_n+q_n$.
We have
\begin{align}
    \max(2tq_n-x_1, 2tq_n-\ell)\leq 7sq_m+(2tq_n-y)\leq 23sq_m,
\end{align}
as well as
\begin{align}
    \min(2tq_n-x_1,2tq_n-\ell)\geq sq_m+(2tq_n-y)\geq 9sq_m,
\end{align}
where we used \eqref{eq:s_qm} to bound $(2tq_n-y)=\mathrm{dist}(y,2q_n\Z)\in [8sq_m, 16sq_m]$.
Hence
\begin{align}\label{eq:x1_ell_1}
    \frac{(x_1-2tq_n)+(\ell-2tq_n)}{2}+4sq_m-3\in [-20sq_m, -5sq_m].
\end{align}
Therefore, when $n$ is large (hence $sq_m$ is large as well), by \eqref{eq:gamma=0_eps}, 
\begin{align}\label{eq:2theta_1_1}
    \|2\theta-1/2+(\frac{(x_1-2tq_n)+(\ell-2tq_n)}{2}+4sq_m-3)\omega\|\geq e^{-20\varepsilon^3 sq_m}\geq e^{-3\varepsilon^3 q_n},
\end{align}
where we used $8sq_m\leq \mathrm{dist}(y,2q_n\Z)\leq q_n$ in the last estimate.

Note that $q_{n+1}\geq e^{\varepsilon q_n}$ and $tq_n\leq 10q_{n+1}^{1-\varepsilon}$, hence
\begin{align}\label{eq:2theta_2_1}
    \|2tq_n\omega\|\leq \frac{2t}{q_{n+1}}\leq 20q_{n+1}^{-\varepsilon}\leq 20e^{-\varepsilon^2 q_n}\ll e^{-3\varepsilon^3 q_n}.
\end{align}
Combining \eqref{eq:2theta_1_1} with \eqref{eq:2theta_2_1}, we have
\begin{align}\label{eq:minmin_22_1}
    &\|2\theta-1/2+(\frac{x_1+\ell}{2}+4sq_m-3)\omega\|\notag\\
    \geq &\|2\theta-1/2+(\frac{(x_1-2tq_n)+(\ell-2tq_n)}{2}+4sq_m-3)\omega\|-\|2tq_n\omega\|\notag\\
    \geq &e^{-21\varepsilon^3 sq_m}.
\end{align}

\underline{Case 2. $x_1\in I_y$, $\ell\in I_0$ or $x_1\in I_0$, $\ell\in I_y$.}
The proof is similar to that of Case 1 above.
Without loss of generality, we assume $y\geq 2(t-1)q_n+q_n$ and $x_1\in I_y$, $\ell\in I_0$.
We have similar to \eqref{eq:x1_ell_1} that
\begin{align}\label{eq:x1_ell_2}
\frac{(x_1-2tq_n)+\ell}{2}+4sq_m-3\in [-11sq_m,-3sq_m].
\end{align}
Hence similar to \eqref{eq:2theta_1_1}, we have
\begin{align}\label{eq:2theta_1_2}
    \|2\theta-1/2+(\frac{(x_1-2tq_n)+\ell}{2}+4sq_m-3)\omega\|\geq e^{-11\varepsilon^3 sq_m}\geq e^{-2\varepsilon^3 q_n}.
\end{align}
Similar to \eqref{eq:2theta_2_1}, we have
\begin{align}
    \|tq_n\omega\|\leq 10e^{-\varepsilon^2 q_n}\ll e^{-2\varepsilon^3 q_n}.
\end{align}
Therefore, similar to \eqref{eq:minmin_22_1}, we have
\begin{align}\label{eq:minmin_22_2}
    \|2\theta-1/2+(\frac{x_1+\ell}{2}+4sq_m-3)\omega\|\geq e^{-12\varepsilon^3 sq_m}.
\end{align}

\underline{Case 3. Both $x_1, \ell \in I_0$.}
We have
\begin{align}
    \frac{x_1+\ell}{2}+4sq_m-3\in [-3sq_m,-sq_m].
\end{align}
Thus
\begin{align}
    \|2\theta-1/2+(\frac{x_1+\ell}{2}+4sq_m-3)\omega\|\geq e^{-3\varepsilon^3 sq_m}.
\end{align}
Combining all these three cases we have proved the claimed result of Lemma \ref{lem:minmin22}.
\end{proof}

The sum of all the $4s$ minimums can be estimated in a similar way as in Lemma \ref{lem:allmin_21}. 
\begin{lemma}\label{lem:allmin_22}
    \begin{align}
        \sum_{j=1}^{4s}\, \min_{\ell\in T_j} \ln \|2\theta-1/2+(\frac{x_1+\ell}{2}+4sq_m-3)\omega\|\geq -(8\beta+16)\varepsilon sq_m.
    \end{align}
\end{lemma}
\begin{proof}[Proof of Lemma \ref{lem:allmin_22}]
For each $w=1,...,4$, we define
\begin{align}
    \tilde{S}_w:=\sum_{j=(w-1)s}^{ws-1}\min_{\ell\in T_j}\ln \|2\theta-1/2+(\frac{x_1+\ell}{2}+4sq_m-3)\omega\|.
\end{align}
Within each $\tilde{S}_w$, the $\ell$'s come from $T^{(w)}$. Hence for two different $\ell,\ell'\in T^{(w)}$, we have, $|\ell-\ell'|<2sq_m<2q_{m+1}$, hence similar to \eqref{eq:x1-ell-x1-ell'}, we have
\begin{align}
    &\|\left(2\theta-1/2+(\frac{x_1+\ell}{2}+4sq_m-3)\omega\right)-\left(2\theta-1/2+(\frac{x_1+\ell'}{2}+4sq_m-3)\omega\right)\|\\
    =&\|\frac{\ell-\ell'}{2}\omega\|\geq \|q_m\omega\|.
\end{align}
This implies, similar to \eqref{eq:Sw} that
\begin{align}
    \tilde{S}_w
    \geq &2\min_{j=1}^{4s}\min_{\ell\in T_j}\ln \|2\theta-1/2+(\frac{x_1+\ell}{2}+4sq_m-3)\omega\|-(2\beta+2)\varepsilon sq_m\\
    \geq &-(2\beta+4)\varepsilon sq_m.
\end{align}
This proves the claimed result of Lemma \ref{lem:allmin_22}. 
\end{proof}
Combining Lemma \ref{lem:allmin_22} with \eqref{eq:sum22}, we have
\begin{align}\label{eq:sum22_nonres}
\sum_{2,2}\geq -4sq_m\ln 2-(8\beta+17)\varepsilon sq_m.
\end{align}
This proves Lemma \ref{lem:sum22} and completes the estimates of $\sum_{2,2}$, hence concludes the proof of Lemma \ref{lem:sum2_nonres} as well.
\end{proof}

Next, we move on and apply Lemma \ref{lem:I1_I2_large_nonres} to establish decay of the eigenfunction.

\subsection{Decay of eigenfunctions}\label{sec:eigen_nonres}

Lemma \ref{lem:I1_I2_large_nonres} implies the following lower bound of denominator over $I_y$.
\begin{lemma}\label{lem:I2_large_nonres}
For $n$ large enough,
\begin{align}\label{eq:P_den_large_nonres}
    \max_{x_1\in I_y} |P_{[x_1,x_1+h-1]}(\theta)|\geq e^{\frac{h}{2}({L_+}-2\varepsilon_0)}.
\end{align}
\end{lemma}
\begin{proof}[Proof of Lemma \ref{lem:I2_large_nonres}]
    It suffices to prove
    \[\max_{x_1\in I_0} |P_{[x_1,x_1+h-1]}(\theta)|<e^{\frac{h}{2}(L_+-2\varepsilon_0)}.\]
    Proof by contradiction, assume there exists $x_1\in I_0$ such that $|P_{[x_1,x_1+h-1]}(\theta)|\geq e^{\frac{h}{2}(L_+-2\varepsilon_0)}$.
    Let $x_2=x_1+h-1$. By the definition of $I_0$, we have
\begin{align}\label{eq:y-x1_y-x2_nonres_0}
sq_m\leq \min(|x_1|, |x_2|)\leq \max(|x_1|,|x_2|)\leq  7sq_m.
\end{align}
Combining Theorem \ref{thm:LE}, Lemma \ref{lem:upperbounds}, with \eqref{eq:rho_bd} and \eqref{eq:Sze2}, we have
\begin{align}\label{eq:nu_1}
\prod_{j=1}^{x_2}|\rho_j|\cdot |P_{[x_1,-1],z}(\theta)|
\leq &\prod_{j=1}^{x_2} |\rho_j| \prod_{j=x_1-1}^{-1} |\rho_j|\cdot  \|S_{-1,z}\cdots S_{x_1-1,z}\|\\
\leq &e^{\frac{h}{2}(L_-+\varepsilon_0)} e^{\frac{|x_1|}{2}(L+\varepsilon_0)}
=e^{\frac{|x_1|}{2}(L_++\varepsilon_0)} e^{\frac{|x_2|}{2}(L_-+\varepsilon_0)},
\end{align}
and similarly
\begin{align}\label{eq:nu_2}
\prod_{j=x_1}^{-1} |\rho_j| \cdot |P_{[1,x_2],z}(\theta)|\leq e^{\frac{|x_2|}{2}(L_++\varepsilon_0)} e^{\frac{|x_1|}{2}(L_-+\varepsilon_0)}.
\end{align}

Finally, combining \eqref{eq:P_den_large_nonres}, \eqref{eq:nu_1}, \eqref{eq:nu_2} with the Poisson formula \eqref{eq:Poisson_1},
\begin{align}\label{eq:nu_3_nonres_0}
|\Psi_0|
\leq &\max(e^{-(\frac{L}{2}-12\varepsilon_0)|x_1|} \max(|\Psi_{x_1-1}|, |\Psi_{x_1}|), e^{-(\frac{L}{2}-12\varepsilon_0)|x_2|} |\max(|\Psi_{x_2}|,|\Psi_{x_2+1}|))\\
\leq &C sq_m e^{-(\frac{L}{2}-12\varepsilon_0)sq_m}
\to 0,
\end{align}
as $n\to\infty$.
The same argument clearly implies $|\Psi_{-1}|\to 0$, which contradictions with $\|(\Psi_0,\Psi_{-1})\|=1$. 
This completes the proof of Lemma \ref{lem:I2_large_nonres}.
\end{proof}

Now we will establish exponential decay of eigenfunction from non-resonant regime to resonant regime.

By Lemma \ref{lem:I2_large_nonres}, for $y$ such that $2(t-1)q_n+q_n^{1-\varepsilon}<y<2tq_n$ and $q_n^{1-\varepsilon}<y<10q_{n+1}^{1-\varepsilon}$, there exists $x_1\in I_y$ such that $|P_{[x_1,x_1+h-1]}(\theta)|\geq e^{\frac{h}{2}(L_+-2\varepsilon_0)}$ holds.
Let $x_2=x_1+h-1$. Similar to \eqref{eq:nu_3_nonres_0}, we have
\begin{align}\label{eq:nu_3_nonres_y}
&\qquad |\Psi_y|\leq \max(e^{-(\frac{L}{2}-12\varepsilon_0)|y-x_1|} \max(|\Psi_{x_1-1}|, |\Psi_{x_1}|), e^{-(\frac{L}{2}-12\varepsilon_0)|y-x_2|} |\max(|\Psi_{x_2}|,|\Psi_{x_2+1}|)).
\end{align}
Also, since $\min(|y-x_1|,|y-x_2|)\geq sq_m\geq \mathrm{dist}(y,2\Z q_n)/8$, we have
\begin{align}\label{eq:nu_3_nonres_y'}
|\Psi_y|\leq e^{-\frac{1}{8}(\frac{L}{2}-12\varepsilon_0)\cdot \mathrm{dist}(y,2\Z q_n)} \max(|\Psi_{x_1-1}|, |\Psi_{x_1}|, |\Psi_{x_2}|, |\Psi_{x_2+1}|).
\end{align}
Rewrite the inequality \eqref{eq:nu_3_nonres_y} as follows, denoting $y=y_0$,
\begin{align}\label{eq:nu_3_new_nonres}
|\Psi_y|\leq \max_{y_1\in \{x_1-1, x_1, x_2, x_2+1\}} e^{-(\frac{L}{2}-12\varepsilon_0)|y_0-y_1|} |\Psi_{y_1}|.
\end{align}
Starting with a  $y=y_0\in \mathcal{N}_{2(t-1),2t}$, we iterate \eqref{eq:nu_3_new_nonres} until we reach at
$y_k$, where either

Case 1. $|y_k-2tq_n| \leq q_n^{1-\varepsilon}$, or 

Case 2. $|t_k-2(t-1)q_n| \leq q_n^{1-\varepsilon}$, or 

Case 3. $k=8q_n^{\varepsilon}$

Hence we have by iterations of \eqref{eq:nu_3_new_nonres} that
\begin{align}\label{eq:nu4_nonres}
|\Psi_y|\leq \max_{(y_1,y_2,...,y_k)\in \mathcal{G}} e^{-(\frac{L}{2}-12\varepsilon_0) \sum_{j=1}^k |y_{j-1}-y_j|} |\Psi_{y_k}|,
\end{align}
where $\mathcal{G}=\{(y_1,y_2,...,y_k):\, y_k \text{ satisfies one of the three cases above}\}$.

If $y_k$ satisfies Case 1, we have by triangle inequality that
\begin{align}\label{eq:nu5_1_nonres}
\sum_{j=1}^k |y_{j-1}-y_j| \geq |y_0-y_k|\geq \mathrm{dist}(y,2\Z q_n)-q_n^{1-\varepsilon}.
\end{align}
And we bound $|\Psi_{y_k}|\leq r_{2t}$.

If $y_k$ satisfies Case 2, we have similarly that
\begin{align}\label{eq:nu5_2_nonres}
\sum_{j=1}^k |y_{j-1}-y_j| \geq \mathrm{dist}(y,2\Z q_n)-q_n^{1-\varepsilon}.
\end{align}
And we bound $|\Psi_{y_k}|\leq r_{2(t-1)}$.

If $y_k$ satisfies Case 3, then since the iteration did not stop due to Case 1 or 2, we have for each $1\leq j\leq k$ that $\mathrm{dist}(y_j, 2\Z q_n)\geq q_n^{1-\varepsilon}$, hence by \eqref{eq:nu_3_nonres_y'},
\begin{align}\label{eq:nu5_3_nonres}
|\Psi_y|
\leq &e^{-\frac{kq_n^{1-\varepsilon}}{8}(\frac{L}{2}-12\varepsilon_0)} \max_{y'\in [2(t-1)q_n+q_n^{1-\varepsilon}, 2tq_n-q_n^{1-\varepsilon}]} |\Psi_{y'}|\notag\\
\leq &e^{-(\frac{L}{2}-12\varepsilon_0)q_n} |\Psi_{y_{\mathrm{max}}}|,
\end{align}
where $y_{\mathrm{max}}$ is such that $\Psi_{y_{\mathrm{max}}}:=\max_{y'\in \mathcal{N}_{2(t-1),2t}} |\Psi_{y'}|$ and we used $k\geq 8q_n^{\varepsilon}$.
Combining the three cases, for any arbitrary $y\in \mathcal{N}_{2(t-1),2t}$, 
\begin{align}\label{eq:Psi_y_decay}
    |\Psi_y|\leq 
    &\max\left(e^{-(\frac{L}{2}-12\varepsilon_0)(\mathrm{dist}(y,2\Z q_n)-q_n^{1-\varepsilon})}\max(r_{2t}, r_{2(t-1)}),\, e^{-(\frac{L}{2}-12\varepsilon_0)q_n} |\Psi_{y_{\mathrm{max}}}|\right)
\end{align}
In particular, this holds for $y=y_{\max}$, yielding
\begin{align}
    |\Psi_{y_{\mathrm{max}}}|\leq e^{-(\frac{L}{2}-12\varepsilon_0)(\mathrm{dist}(y_{\mathrm{max}},2\Z q_n)-q_n^{1-\varepsilon})}\max(r_{2t}, r_{2(t-1)})\leq \max(r_{2t},r_{2(t-1)}).
\end{align}
Plugging this into \eqref{eq:Psi_y_decay} yields
the claimed result of Theorem \ref{lem:non-res}. \qed

\section{Eigenfunction in the resonant regimes in the strong Liouville scale}\label{sec:res}
\begin{theorem}\label{lem:r2t}
For any $t\in \Z$ such that $0\neq |t|<2q_n^{-1}q_{n+1}^{1-\varepsilon}$, we have
\[
r_{2t}\leq e^{-(L-\beta(\omega)-26\varepsilon_0)\,tq_n}.
\]
\end{theorem}
This whole section is devoted to the proof of this theorem.

Let $h=4q_n-2$ and set 
\begin{align*}
\begin{cases}
I_0:= [-3q_n, -q_n-1]\cap (2\Z+1), \\
I_y:=[y-3q_n, y-q_n-1]\cap (2\Z+1).
\end{cases}
\end{align*}
Note that $|I_0|=|I_y|=q_n$. Hence $|I_0\cup I_y|=2q_n=h/2+1$.

The first lemma we will prove is the following:
\begin{lemma}\label{lem:I1_I2_large_res}
There exists $j\in I_0\cup I_y$ such that 
    \begin{align}\label{eq:P_den_large_res}
         |P_{[j,j+h-1]}(\theta)|\geq e^{2q_n(L_+-\frac{\beta}{2}-2\varepsilon_0)}.
    \end{align}
\end{lemma}
\begin{proof}[Proof of Lemma \ref{lem:I1_I2_large_res}]
Let $\theta_{\ell} := \theta+\frac{\ell-1}{2}\omega+\frac{h/2-1}{2}\omega-\frac{1}{4}$.
Proof by contradiction, assume 
\begin{align}
    \max_{j\in I_0\cup I_y} |P_{[j,j+h-1]}(\theta)|<e^{2q_n(L_+-\frac{\beta}{2}-2\varepsilon_0)}.
\end{align}
Then, similar to \eqref{eq:P<U_non_res},
\begin{align}
    \sup_{\xi\in [-1,1]}|g_{2q_n-1}(\xi)|\leq 2q_n e^{2q_n(L_+-\frac{\beta}{2}-2\varepsilon_0)} \cdot  \sup_{\xi\in [-1,1]} \max_{j\in I_0\cup I_y} |U(\xi, j, 4q_n-2, I_0\cup I_y)|.
\end{align}
Combining this with the following Lemma \ref{lem:uni_res}, we arrive at a contradiction to Lemma \ref{lem:ave_low}.
\begin{lemma}\label{lem:uni_res}
For $n$ large enough, 
    \begin{align}
        \sup_{\tilde{\theta}\in \T} \max_{j\in I_0\cup I_y} |U(\cos(2\pi\tilde{\theta}), j, 4q_n-2, I_0\cup I_y)|\leq e^{(\frac{\beta}{2}+\varepsilon_0)2q_n}.
    \end{align}
\end{lemma}
Therefore, Lemma \ref{lem:I1_I2_large_res} is proved up to the proof of Lemma \ref{lem:uni_res}.
\end{proof}

\subsection{Proof of Lemma \ref{lem:uni_res}}\label{sec:uni_res}
Without loss of generality we assume $y$ is even and $sq_m$ is even. 
\begin{small}
\begin{align}\label{eq:sum1-sum2_res}
    &\ln |U(\cos(2\pi\tilde{\theta}),j,4q_n-2,I_0\cup I_y)| \notag\\
    =& \ln \prod_{\substack{\ell\in I_0\cup I_y\\ \ell \neq j}} \left| \frac{\cos 2\pi\tilde{\theta}-\cos 2\pi \theta_\ell}{\cos 2\pi \theta_j-\cos 2\pi \theta_\ell}\right|\notag\\
    =&\sum_{\substack{\ell\in I_0\cup I_y\\ \ell\neq j}}\ln |\cos2\pi \tilde{\theta}-\cos 2\pi \theta_\ell| -\sum_{\substack{\ell\in I_0\cup I_y\\ \ell\neq j}} \ln |\cos 2\pi \theta_j-\cos 2\pi \theta_\ell|\notag\\
=:&\sum_1-\sum_2,
\end{align}
\end{small}
in which $\tilde{\theta}\in \T$ be arbitrary and $j\in I_0\cup I_y$.
We have by Lemma \ref{lem:nonmin}, 
\begin{align}\label{eq:sum1_res}
\sum_1=&(2q_n-1)\ln 2 \notag\\
&+\sum_{\substack{\ell\in I_0\cup I_y\\ \ell\neq j}}\ln |\sin(\pi(\tilde{\theta}-\theta_{\ell}))|+\sum_{\substack{\ell\in I_0\cup I_y\\ \ell\neq j}} \ln |\sin(\pi(\tilde{\theta}+\theta_{\ell}))| \notag\\
\leq &2q_n \ln 2 +4(C\ln q_n-(q_n-1)\ln 2)\notag\\
\leq &2q_n(-\ln 2+\varepsilon).
\end{align}

The main difficulty lies in obtaining a sharp lower bound of $\sum_2$.
\begin{align}\label{eq:sum2=sum2122_res}
\sum_2=&(2q_n-1)\ln 2+\sum_{\substack{\ell\in I_0\cup I_y\\ \ell\neq j}}\ln|\sin(\pi\frac{j-\ell}{2}\omega)|\notag\\
&+
\sum_{\substack{\ell\in I_0\cup I_y\\ \ell\neq j}}\ln|\sin(\pi(2\theta-\frac{1}{2}+(\frac{j+\ell}{2}+2q_n-3)\omega))| \notag\\
=:&(2q_n-1)\ln 2+\sum_{2,1}+\sum_{2,2}.
\end{align}

\subsubsection{Estimates of $\sum_{2,1}$}
We first prove the following estimate regarding $\sum_{2,1}$.
\begin{lemma}\label{lem:21_res}
    \begin{align}
        \sum_{2,1}\geq (-2\ln 2-\beta-\varepsilon)q_n.
    \end{align}
\end{lemma}
\begin{proof}[Proof of Lemma \ref{lem:21_res}]
Without loss of generality, we assume $j\in I_y$.
We have by Lemma \ref{lem:nonmin} that
\begin{align}\label{eq:sum21_res}
\sum_{2,1}
&=\sum_{\substack{\ell\in I_0\cup I_y\\ \ell\neq j}}\ln|\sin(\pi(\frac{j-\ell}{2}\omega)|\notag\\
&\geq 2 (-q_n\ln 2-C\ln q_n)+ \min_{\ell\in I_0}\ln\|\frac{j-\ell}{2}\omega\|.
\end{align}
Since $j\in I_y$ and $|y-2tq_n|<q_n^{1-\varepsilon}$, we have
\begin{align}
    j-2tq_n\in [y-2tq_n-3q_n-1, y-2tq_n-q_n-1]\subset [-(3+\varepsilon)q_n-1, (\varepsilon-1)q_n-1].
\end{align}
Thus there exists a unique $t'\in \{t-1,t,t+1\}$ such that $j-2t'q_n\in I_0$.
We estimate, using $t'q_n<20q_{n+1}^{1-\varepsilon}$ and $q_{n+1}\geq e^{\varepsilon q_n}$, that
\begin{align}\label{eq:t'qn_small}
    \|\frac{j-(j-2t'q_n)}{2}\omega\|=\|t'q_n\omega\|\leq \frac{t'}{q_{n+1}}<20q_{n+1}^{-\varepsilon}\leq 20e^{-\varepsilon^2 q_n}.
\end{align}
Also if $\ell\in I_0\setminus \{j-2t'q_n\}$, we have $0<|\ell-(j-2t'q_n)|<2q_n$, hence
\begin{align}\label{eq:t'qn-ell_large}
    \|\frac{\ell-(j-2t'q_n)}{2}\omega\|\geq \|q_{n-1}\omega\|\geq \frac{1}{2q_n}\gg 20e^{-\varepsilon^2 q_n}.
\end{align}
Combining \eqref{eq:t'qn_small} with \eqref{eq:t'qn-ell_large}, we have for any $\ell\in I_0\setminus \{j-2t'q_n\}$ that 
\begin{align}
    \|\frac{j-\ell}{2}\omega\|\geq \|\frac{\ell-(j-2t'q_n)}{2}\omega\|-\|t'q_n\omega\|\geq \frac{1}{3q_n}\gg \|\frac{j-(j-2t'q_n)}{2}\omega\|.
\end{align}
Hence
\begin{align}
    \min_{\ell\in I_0}\ln \|\frac{j-\ell}{2}\omega\|=\ln \|\frac{j-(j-2t'q_n)}{2}\omega\|=\ln \|t'q_n\omega\|\geq \ln \frac{t'}{2q_{n+1}}\geq -(\beta+\varepsilon) q_n,
\end{align}
where we used $(\ln q_{n+1})/q_n<\beta+\varepsilon/2$.
Plugging this into \eqref{eq:sum21_res} yields the claimed result of Lemma \ref{lem:21_res}.
\end{proof}

\subsubsection{Estimates of $\sum_{2,2}$}

The next lemma concerns $\sum_{2,2}$.
\begin{lemma}\label{lem:22_res}
    \begin{align}
        \sum_{2,2}\geq (-2\ln 2-2\varepsilon^3)q_n.
    \end{align}
\end{lemma}
\begin{proof}[Proof of Lemma \ref{lem:22_res}]
By Lemma \ref{lem:nonmin},
\begin{align}\label{eq:sum22_res}
\sum_{2,2}
=&\sum_{\substack{\ell\in I_0\cup I_y\\ \ell\neq j}}\ln|\sin(\pi(2\theta-\frac{1}{2}+(\frac{j+\ell}{2}+q_n-3)\omega))|\notag\\
\geq & 2(-q_n\ln 2-\varepsilon) + 2\min_{\{\ell,j\}\subset I_0\cup I_y}\ln \|2\theta-\frac{1}{2}+(\frac{j+\ell}{2}+2q_n-3)\omega\|.
\end{align}

Similar to the proof of Lemma \ref{lem:minmin22}, we divide into three cases, depending on the locations of $\ell$ and $j$.

\underline{Case 1. Both $\ell,j\in I_y$}.
We have, again by using $tq_n<10q_{n+1}^{1-\varepsilon}$ and $q_{n+1}\geq e^{\varepsilon q_n}$, 
\begin{align}\label{eq:2t+2_1}
    \|2(t+1)q_n\omega\|\leq \frac{2(t+1)}{q_{n+1}}<30e^{-\varepsilon^2 q_n}.
\end{align}
Also since
\begin{align}\label{eq:1111}
    \frac{(j-2(t+1)q_n)+(\ell-2(t+1)q_n)}{2}+2q_n-3\in [-4q_n, -q_n],
\end{align}
we have by $\tilde{\gamma}(\omega,\theta)=0$ that 
\begin{align}\label{eq:2t+2_2}
    &\quad \|2\theta-\tfrac{1}{2}+(\frac{(j-2(t+1)q_n)+(\ell-2(t+1)q_n)}{2}+2q_n-3)\omega\|\geq e^{-\varepsilon^3 q_n}\gg 30e^{-\varepsilon^2 q_n}.
\end{align}
Combining \eqref{eq:2t+2_1} with \eqref{eq:2t+2_2}, we have
\begin{align}
    &\|2\theta-1/2+(\frac{j+\ell}{2}+2q_n-3)\omega\|\\
    \geq &\|2\theta-1/2+(\frac{(j-2(t+1)q_n)+(\ell-2(t+1)q_n)}{2}+2q_n-3)\omega\|-\|2(t+1)q_n\omega\|\\
    \geq &\frac{1}{2}e^{-\varepsilon^3 q_n}.
\end{align}

\underline{Case 2. $\ell\in I_y$, $j\in I_0$ or $\ell\in I_0$, $j\in I_y$}.
Without loss of generality, we assume $\ell\in I_y$ and $j\in I_0$. 
Similar to \eqref{eq:2t+2_1}, 
\begin{align}\label{eq:t+2_1}
    \|(t+2)q_n\omega\|\leq \frac{t+2}{q_{n+1}}\leq 20e^{-\varepsilon^2 q_n}.
\end{align}
Similar to \eqref{eq:1111},
\begin{align}
    \frac{(j-2q_n)+(\ell-2(t+1)q_n)}{2}+2q_n-3\in [-4q_n, -q_n],
\end{align}
hence 
\begin{align}\label{eq:2222}
    \|2\theta-1/2+\frac{(j-2q_n)+(\ell-2(t+1)q_n)}{2}+2q_n-3)\omega\|\geq e^{-\varepsilon^3 q_n}\gg 20e^{-\varepsilon^2 q_n}.
\end{align}
Combining \eqref{eq:t+2_1} with \eqref{eq:2222}, we have
\begin{align}
    &\|2\theta-1/2+\frac{j+\ell}{2}+2q_n-3)\omega\|\\
    \geq &\|2\theta-1/2+\frac{(j-2q_n)+(\ell-2(t+1)q_n)}{2}+2q_n-3)\omega\|-\|(t+2)q_n\omega\|\\
    \geq &\frac{1}{2} e^{-\varepsilon^3 q_n}.
\end{align}

\underline{Case 3. Both $\ell,j\in I_0$}. The analysis is very similar to the previous two cases, we just point out that similar to \eqref{eq:2222}, we estimate
\begin{align}\label{eq:3333}
    \|2\theta-1/2+\frac{(j-2q_n)+(\ell-2q_n)}{2}+2q_n-3)\omega\|\geq e^{-\varepsilon^3 q_n}
\end{align}
We leave the details to interested readers.

Finally, combining the three cases, we have proved Lemma \ref{lem:22_res}.
\end{proof}

Combining Lemmas \ref{lem:21_res}, \ref{lem:22_res} with \eqref{eq:sum2=sum2122_res}, we have
\begin{align}\label{eq:sum22_2_res}
\sum_{2} \geq (-2\ln 2-\beta-2\varepsilon) q_n.
\end{align}
Finally combining \eqref{eq:sum22_2_res} with \eqref{eq:sum1-sum2_res} and \eqref{eq:sum1_res}, we have
\begin{align}\label{eq:sum_res_final}
\sum_1-\sum_2 \leq (\frac{\beta}{2}+2\varepsilon) 2q_n.
\end{align}
This completes the proof of Lemma \ref{lem:uni_res}. \qed

\subsection{Decay of eigenfunctions}

Lemma \ref{lem:I1_I2_large_res} implies the following.
\begin{lemma}\label{lem:I2_large_res}
There exists $j\in I_y$ such that \eqref{eq:P_den_large_res} holds.
\end{lemma}
\begin{proof}[Proof of Lemma \ref{lem:I2_large_res}]
It suffices to prove that for any $j\in I_0$, $|P_{[j,j+h-1]}(\theta)|<e^{2q_n(L_+-\frac{\beta}{2}-2\varepsilon_0)}$.
Proof by contradiction. Suppose that for some $j\in I_0$, $|P_{[j,j+h-1]}(\theta)|\geq e^{2q_n(L_+-\frac{\beta}{2}-2\varepsilon_0)}$.
Let $x_2=j+4q_n-2$. 
Note that by the construction of $I_0, I_y$, and by \eqref{eq:s_qm}, we have
\begin{align}\label{eq:x1x2}
\min (|j|, |x_2|)\geq q_n-2
\end{align}

Combining Theorem \ref{thm:LE}, Lemma \ref{lem:upperbounds} with \eqref{eq:int_rho}, \eqref{eq:scalar_upper} and \eqref{eq:Sze2}, we have that for $n$ sufficiently large, the upper bounds of numerators are
\begin{align}\label{eq:nu_1_0}
\prod_{j=1}^{x_2}|\rho_j|\cdot |P_{[x_1,-1],z}(\theta)|
\leq &\prod_{j=1}^{x_2} |\rho_j| \prod_{j=x_1-1}^{-1} |\rho_j|\cdot  \|S_{-1,z}\cdots S_{x_1-1,z}\|\\
\leq &e^{\frac{h}{2}(L_-+\varepsilon_0)} e^{\frac{-x_1}{2}(L+\varepsilon_0)},
\end{align}
and similarly
\begin{align}\label{eq:nu_2_0}
\prod_{j=x_1}^{-1} |\rho_j| \cdot |P_{[1,x_2],z}(\theta)|\leq e^{\frac{h}{2}(L_-+\varepsilon_0)} e^{\frac{x_2}{2}(L+\varepsilon_0)}.
\end{align}
Combining them with \eqref{eq:Poisson_1}, 
we have
\begin{align}\label{eq:Psi0_0}
    |\Psi_0|\leq \max\left(e^{-(\frac{L}{2}-12\varepsilon_0)|j|+\beta q_n}\max(|\Psi_{j}|,|\Psi_{j-1}|), \right.\\ 
    \left. e^{-(-\frac{L}{2}-12\varepsilon_0)|x_2|+\beta q_n}\max(|\Psi_{x_2}|,|\Psi_{x_2+1}|)\right).
\end{align}

We divide the following discussions into three cases.

\underline{Case 1.  $j \in [-2q_n-q_n^{1-\varepsilon}, -2q_n+q_n^{1-\varepsilon}]$.}
In this case, $x_2\in [2q_n-q_n^{1-\varepsilon}, 2q_n+q_n^{1-\varepsilon}]$.
We can then bound directly that $\max(|\Psi_j|,|\Psi_{j-1}|)\leq r_{-2}$ and $\max(|\Psi_{x_2}|, |\Psi_{x_2+1}|)\leq r_2$. Then
\begin{align}\label{eq:Psi_0_1}
    |\Psi_0|
    \leq &e^{-(\frac{L}{2}-12\varepsilon_0)(2q_n-q_n^{1-\varepsilon})+\beta q_n} \max(r_{-2}, r_2)\notag\\
    \leq &e^{-(L-\beta-25\varepsilon_0)q_n}\to 0,
\end{align}
as $n\to\infty$.

Case 2,  $j\in [-3q_n, -2q_n-q_n^{1-\varepsilon})$.
In this case, we apply Theorem \ref{lem:non-res} to $|\Psi_j|$ and $|\Psi_{j-1}|$ and obtain
\begin{align}
    \max(|\Psi_j|, |\Psi_{j-1}|)\leq e^{-(\frac{L}{2}-12\varepsilon_0)(\mathrm{dist}(j,2\Z q_n)-q_n^{1-\varepsilon})} \max(r_{-4},r_{-2}).
\end{align}
Similarly
\begin{align}
    \max(|\Psi_{x_2}|,|\Psi_{x_2+1}|)\leq e^{-(\frac{L}{2}-12\varepsilon_0)(\mathrm{dist}(x_2,2\Z q_n)-q_n^{1-\varepsilon})} \max(j_0,r_2).
\end{align}
Note that $x_2\in [q_n,2q_n-q_n^{1-\varepsilon}]$.
Plugging this into \eqref{eq:Psi0_0}, we have
\begin{align}\label{eq:Psi_0_2}
    |\Psi_0|
    \leq &\max(e^{-(\frac{L}{2}-12\varepsilon_0)(4q_n-q_n^{1-\varepsilon})+\beta q_n}r_{-4}, e^{-(\frac{L}{2}-12\varepsilon_0)(2q_n-q_n^{1-\varepsilon})+\beta q_n}r_{-2},\notag\\
    &\qquad\qquad e^{-(\frac{L}{2}-12\varepsilon_0)(2q_n-q_n^{1-\varepsilon})+\beta q_n}j_0, e^{-(\frac{L}{2}-12\varepsilon_0)(2q_n-q_n^{1-\varepsilon})+\beta q_n}r_2)\notag\\
    \leq &\max(e^{-(2L-\beta-50\varepsilon_0)q_n}r_{-4}, e^{-(L-\beta-25\varepsilon_0)q_n}\max(r_{-2},j_0,r_2)\notag\\
    \leq &e^{-(L-\beta-26\varepsilon_0)q_n}\to 0.
\end{align}

\underline{Case 3. $j\in (-2q_n+q_n^{1-\varepsilon}, -q_n-1]$}

Similar to Case 2, one can also get
\begin{align}\label{eq:Psi_0_3}
|\Psi_0| \leq &\max(e^{-(2L-\beta-50\varepsilon_0)q_n}r_{-4}, e^{-(L-\beta-25\varepsilon_0)q_n}\max(r_{-2},j_0,r_2)\notag\\
    \leq &e^{-(L-\beta-26\varepsilon_0)q_n}\to 0.
\end{align}

The exact same estimates in all three cases hold for $\Psi_{-1}$ in place of $\Psi_0$. Hence we get $\|(\Psi_0,\Psi_{-1})\|\to 0$, which contradicts with $\|(\Psi_0,\Psi_{-1})\|=1$.
This completes the proof of Lemma \ref{lem:I2_large_res}.
\end{proof}

Now we are in place to conclude the whole proof of Anderson localization in this strong Liouville scale.
The same argument in the proof of Lemma \ref{lem:I2_large_res} implies for any $0\neq t\in \Z$ such that $2|t|q_n<10q_{n+1}^{1-\varepsilon}$, 
\begin{align}\label{eq:r2t<r2tpm2}
    r_{2t}\leq &\max(e^{-(L-\frac{\beta}{2}-25\varepsilon_0)2q_n}\max(r_{2(t-2)}, r_{2(t+2)}), \\
    &\qquad\qquad e^{-(L-\beta-25\varepsilon_0)q_n}\max(r_{2(t-1)},r_{2t},r_{2(t+1)}).
\end{align}
Clearly, the $r_{2t}$ term on the right hand side can be dropped. 
Now if we start with $t$ such that $0\neq |t|<2q_n^{-1}q_{n+1}^{1-\varepsilon}$, we can iterate \eqref{eq:r2t<r2tpm2} until either one of $t-2,t+2,t-1,t+1$ becomes zero, or one of them goes beyond $5q_n^{-1}q_{n+1}^{1-\varepsilon}=:t_n$.
Therefore
\begin{align}
    r_{2t}
    \leq &\max(e^{-(L-\beta-25\varepsilon_0)tq_n} j_0, Ce^{-(L-\beta-25\varepsilon_0)(t_n-t)q_n} q_{n+1}^{1-\varepsilon})\\
    \leq &\max(e^{-(L-\beta-26\varepsilon_0)tq_n}, e^{-(L-\beta-26\varepsilon_0)3q_{n+1}^{1-\varepsilon}})=e^{-(L-\beta-26\varepsilon_0)tq_n}.
\end{align}
This completes the proof of Theorem \ref{lem:r2t}, hence the whole proof of localization. \qed

\section{Eigenfunction in a weak Liouville scale}\label{sec:Diophantine}
For $q_n^{1-\varepsilon}\leq y\leq 10q_{n+1}^{1-\varepsilon}$ and $q_{n+1}\leq e^{\varepsilon q_n}$, we do not need to divide into resonant/non-resonant regimes. The key input is an upper bound on $U$, see Lemma \ref{lem:non_res_uni_Dio} below. We only sketch this part of the argument here and leave further details to interested readers. Instead, let $m$ be the largest integer such that 
$y\geq 8q_m$ and choose the largest integer $s$ such that $8sq_m\leq y<8(s+1)sq_m$. 
The following inequalities will be used throughout this section:
\begin{align}\label{eq:s_qm_Dio}
    \max(8sq_m, q_n^{1-\varepsilon})\leq y<\min (16sq_m, 8q_{m+1}).
\end{align}
Let $h=8sq_m-2$ and set $I_0,I_y$ as in \eqref{def:I0Iy}.

Clearly the key is to prove the following upper bound of $U$. We only sketch the proof of this lemma in the rest of the section and leave the study of $\Psi_y$ to interested readers.

\begin{lemma}\label{lem:non_res_uni_Dio}
We have
\begin{align}
        \sup_{\tilde{\theta}\in \T}\max_{x_1\in I_0\cup I_y} |U(\cos(2\pi\tilde{\theta}),x_1,h,I_0\cup I_y)|\leq e^{\frac{\varepsilon_0 h}{2}}.
\end{align}
\end{lemma}
\begin{proof}[Proof of Lemma \ref{lem:non_res_uni_Dio}]
Without loss of generality we assume $y$ is even and $7sq_m$ is even. For each $\ell\in I_0\cup I_y$, let 
$\theta_{\ell}=\theta+\frac{\ell}{2}\omega+(2sq_m-\frac{3}{2})\omega-\frac{1}{4}$.

As in \eqref{eq:U_sum1sum2}, we decompose
\begin{align}
    \ln |U(\cos(2\pi\tilde{\theta}),x_1, h,I_0\cup I_y)|
=:\sum_1-\sum_2.
\end{align}
Similar to \eqref{eq:sum1_nonres}, we have
\begin{align}\label{eq:sum1_nonres_Dio}
\sum_1\leq 4sq_m(-\ln 2+\varepsilon).
\end{align}
As in \eqref{eq:sum2_nonres}, we decompose
\begin{align}\label{eq:sum2_nonres_Dio}
\sum_2=:(4sq_m-1)\ln 2+\sum_{2,1}+\sum_{2,2}.
\end{align}
We further break $I_0, I_y$ into $T_j$ and $T^{(w)}$, which are defined the same way as in \eqref{def:Tj}, \eqref{def:T^w}.

\subsection{Estimates of $\sum_{2,1}$}
\begin{lemma}\label{lem:sum21_Dio}
    \begin{align}
        \sum_{2,1}\geq -4sq_m\ln 2-(8\beta+600)\varepsilon sq_m.
    \end{align}
\end{lemma}
\begin{proof}[Proof of Lemma \ref{lem:sum21_Dio}]
Without loss of generality, we assume $x_1\in T_{j_0}$ for some $j_0\leq 3s$. Note $x_1\in I_y$.
We have by Lemma \ref{lem:nonmin}, similar to \eqref{eq:sum21_nonres} that
\begin{align}\label{eq:sum21_nonres_Dio}
\sum_{2,1}\geq -4sq_m\ln 2-\varepsilon sq_m+\sum_{\substack{r=1\\ r\neq j_0}}^{4s} \min_{\ell\in T_r}\ln\|\frac{x_1-\ell}{2}\omega\|
\end{align}

Let 
 \[\min_{\substack{r=1\\ r\neq j_0}}^{4s} \min_{\ell\in T_r}\ln\|\frac{x_1-\ell}{2}\omega\| =: \ln\|\frac{x_1-\hat{\ell}}{2}\omega\|\]
be the minimum of the $(4s-1)$ minimums.
\begin{lemma}\label{lem:minmin_21_Dio}
The minimum of the $(4s-1)$ minimums satisfies the following lower bound:
    \begin{align}
        \ln \|\frac{x_1-\hat{\ell}}{2}\omega\|\geq -64\varepsilon sq_m.
    \end{align}
\end{lemma}
\begin{proof}[Proof of Lemma \ref{lem:minmin_21_Dio}] We divide into two cases depending on the largeness of $y$.

\underline{Case 1. If $q_n^{1-\varepsilon}\leq y\leq q_n/2$.}
We have
\begin{align}
    0<|\frac{x_1-\hat{\ell}}{2}|<y<\frac{q_n}{2}.
\end{align}
Hence
\begin{align}
    \ln\|\frac{x_1-\hat{\ell}}{2}\omega\|\geq \ln\|q_{n-1}\omega\|\geq \ln \frac{1}{2q_n}\geq -\varepsilon sq_m,
\end{align}
where we used $16sq_m>q_n^{1-\varepsilon}$ in the last estimate.

\underline{Case 2. If $q_n/2\leq y\leq 10q_{n+1}^{1-\varepsilon}$.}
We have
\begin{align}
0<|\frac{x_1-\hat{\ell}}{2}|<q_{n+1},
\end{align}
hence
\begin{align}
    \ln \|\frac{x_1-\hat{\ell}}{2}\|\geq \ln \|q_n\omega\|\geq \ln \frac{1}{2q_{n+1}}\geq -2\varepsilon q_n\geq -64\varepsilon sq_m.
\end{align}
where we used $q_{n+1}\leq e^{\varepsilon q_n}$ and $q_n/2\leq y\leq 16sq_m$, see \eqref{eq:s_qm_Dio}. This completes the proof of Lemma \ref{lem:minmin_21_Dio}.
\end{proof}

Next, similar to Lemma \ref{lem:allmin_21}, we estimate the sum of the minimums as follows.
\begin{lemma}\label{lem:allmin_21_Dio}
    \begin{align}
        \sum_{\substack{r=1\\ r\neq j_0}}^{4s}\, \min_{\ell\in T_r} \ln \|\frac{x_1-\ell}{2}\omega\|\geq -(8\beta+550)\varepsilon sq_m.
    \end{align}
\end{lemma}
\begin{proof}
For each $w=1,...,4$, we define
\begin{align}
    S_w:=\sum_{\substack{r=(w-1)s+1\\ r\neq j_0}}^{ws} \min_{\ell\in T_r} \ln \|\frac{x_1-\ell}{2}\omega\|.
\end{align}
Clearly,
\begin{align}
    \sum_{\substack{r=1\\ r\neq j_0}}^{4s}\, \min_{\ell\in T_r} \ln \|\frac{x_1-\ell}{2}\omega\|
    =&\sum_{w=1}^4 S_w.
\end{align}
Within each sum $S_w$, the $\ell$'s are chosen from the same $T^{(w)}$, see \eqref{def:T^w}, hence for arbitrary two different $\ell, \ell'\in T^{(w)}$, $|\ell-\ell'|<2sq_m<2q_{m+1}$. 
Therefore
\begin{align}\label{eq:x1-ell-x1-ell'_Dio}
    \|\frac{x_1-\ell}{2}\omega-\frac{x_1-\ell'}{2}\omega\|\geq \|q_m\omega\|.
\end{align}
This way, we can estimate
\begin{align}
    S_w
    \geq &2\min_{\substack{r=1\\ r\neq j_0}}^{4s} \min_{\ell\in T_r} \ln \|\frac{x_1-\ell}{2}\omega\|+2\sum_{j=1}^s \ln (j\|q_m\omega\|)\\
    \geq &-128\varepsilon sq_m+2s\ln \frac{s}{q_{m+1}},
\end{align}
where we used Lemma \ref{lem:minmin_21_Dio} for the minimum of minimum terms.
Next we continue to lower bound $s\ln (s/q_{m+1})$ as follows, dividing into two cases. 

\underline{Case 1. If $q_n^{1-\varepsilon}\leq y\leq q_n/2$}.
In this case, we have $8q_m\leq y<q_n/2$, implying $q_{m+1}\leq q_n$. We also have $16sq_m>y\geq q_n^{1-\varepsilon}\geq q_{m+1}^{1-\varepsilon}$ which implies
\[s\ln\frac{s}{q_{m+1}}\geq s\ln\frac{q_{m+1}^{1-\varepsilon}}{16q_mq_{m+1}}\geq -(\beta+2)\varepsilon sq_m,\]
where we used $(\ln q_{m+1})/q_m\leq \beta+\varepsilon$ in the last estimate.
Hence each 
\begin{align}\label{eq:Sw_Dio}
S_w\geq -(2\beta+132)\varepsilon sq_m.
\end{align}

\underline{Case 2. If $q_n/2<y<10q_{n+1}^{1-\varepsilon}$}. 
In this case, we have $8q_m<y<q_{n+1}$.
We further divide into:

\underline{Case 2.1. If $m=n$}.
Recall that $q_{n+1}\leq e^{\varepsilon q_n}$, we can bound
\begin{align}
    s\ln \frac{s}{q_{m+1}}\geq s\ln \frac{s}{q_{n+1}}\geq -\varepsilon sq_n=-\varepsilon sq_m.
\end{align}
Hence 
\begin{align}\label{eq:Sw_Dio_21}
    S_w\geq -130\varepsilon sq_m.
\end{align}

\underline{Case 2.2. If $m<n$}. In this case, $q_{m+1}\leq q_n$, hence we can bound
\begin{align}
    \ln \frac{s}{q_{m+1}}\geq \ln \frac{s}{q_n}>\ln \frac{1}{32q_m}>-\varepsilon q_m.
\end{align}
where we used $16sq_m\geq y\geq q_n/2$.
Hence we again have 
\begin{align}\label{eq:Sw_Dio_22}
    S_w\geq -130\varepsilon sq_m.
\end{align}
Combining \eqref{eq:Sw_Dio}, \eqref{eq:Sw_Dio_21} and \eqref{eq:Sw_Dio_22}, we have proved Lemma \ref{lem:allmin_21_Dio}.
\end{proof}
Combining Lemma \ref{lem:allmin_21_Dio} into \eqref{eq:sum21_nonres_Dio}, we have 
\begin{align}\label{eq:sum21_Dio}
\sum_{2,1} &\geq -4sq_m\ln 2-\varepsilon s q_m+\sum_{\substack{r=1\\ r\neq j_0}}^{4s} \min_{\ell\in T_r}\ln\|\frac{x_1-\ell}{2}\omega\|\notag\\ 
&\geq  -4sq_m\ln 2-(8\beta+600)\varepsilon sq_m.
\end{align}
This completes the proof of Lemma \ref{lem:sum21_Dio}.
\end{proof}

Next, we move on to estimate $\sum_{2,2}$.

\subsection{Estimates of $\sum_{2,2}$}
\begin{lemma}\label{lem:sum22_Dio}
    \begin{align}
        \sum_{2,2}\geq -4sq_m\ln 2-(8\beta+200)\varepsilon sq_m.
    \end{align}
\end{lemma}
\begin{proof}[Proof of Lemma \ref{lem:sum22_Dio}]
Similar to \eqref{eq:sum21_nonres_Dio}, we have
\begin{align}\label{eq:sum22_Dio}
\sum_{2,2}
\geq &-4sq_m\ln 2-\varepsilon sq_m+\sum_{r=1}^{4s} \min_{\ell\in T_r}\ln\|2\theta-1/2+(\frac{x_1+\ell}{2}+4sq_m-3)\omega)\|
\end{align}
We again estimate the minimum of the $4s$ minimum terms first.
\begin{lemma}\label{lem:minmin22_Dio}
    \begin{align}
        \min_{r=1}^{4s}\min_{\ell\in T_r}\ln\|2\theta-1/2+(\frac{x_1+\ell}{2}+4sq_m-3)\omega)\|\geq -20\varepsilon sq_m.
    \end{align}
\end{lemma}
\begin{proof}[Proof of Lemma \ref{lem:minmin22_Dio}]
We divide into three different cases, depending on the locations of $x_1$ and $\ell$.

\underline{Case 1. Both $x_1, \ell \in I_y$}.
\begin{align}
    \frac{x_1+\ell}{2}+4sq_m-3\in [y-3sq_m,y+3sq_m]\subset [5sq_m,19sq_m],
\end{align}
where we used $8sq_m\leq y<16sq_m$.

\underline{Case 2. $x_1\in I_y$, $\ell\in I_0$ or $x_1\in I_0$, $\ell\in I_y$}.
\begin{align}
    \frac{x_1+\ell}{2}+4sq_m-3\in [\frac{y}{2}-3sq_m, \frac{y}{2}+sq_m]\subset [sq_m,9sq_m].
\end{align}

\underline{Case 3. Both $x_1,\ell\in I_0$}.
\begin{align}
    \frac{x_1+\ell}{2}+4sq_m-3\in [-3sq_m, -sq_m].
\end{align}

Hence in all three cases, we have
\begin{align}
    \|2\theta-1/2+(\frac{x_1+\ell}{2}+4sq_m-3)\omega\|\geq e^{-20\varepsilon sq_m},
\end{align}
as claimed, and thus completed the proof of Lemma \ref{lem:minmin22_Dio}.
\end{proof}

The other minimums can be estimated in a similar way as in Lemma \ref{lem:allmin_21_Dio}. 
\begin{lemma}\label{lem:allmin_22_Dio}
    \begin{align}
        \sum_{r=1}^{4s}\, \min_{\ell\in T_r} \ln \|2\theta-1/2+(\frac{x_1+\ell}{2}+4sq_m-3)\omega\|\geq -(8\beta+180)\varepsilon sq_m.
    \end{align}
\end{lemma}
\begin{proof}[Proof of Lemma \ref{lem:allmin_22_Dio}]
For each $w=1,...,4$, we define
\begin{align}
    \tilde{S}_w:=\sum_{r=(w-1)s}^{ws-1}\min_{\ell\in T_r}\ln \|2\theta-1/2+(\frac{x_1+\ell}{2}+4sq_m-3)\omega\|.
\end{align}
Within each $\tilde{S}_w$, the $\ell$'s come from the same $T^{(w)}$. Hence for two different $\ell,\ell'\in T^{(w)}$, we have, $|\ell-\ell'|<2sq_m<2q_{m+1}$, hence similar to \eqref{eq:x1-ell-x1-ell'_Dio}, we have
\begin{align}
    &\|\left(2\theta-1/2+(\frac{x_1+\ell}{2}+4sq_m-3)\omega\right)-\left(2\theta-1/2+(\frac{x_1+\ell'}{2}+4sq_m-3)\omega\right)\|\\
    =&\|\frac{\ell-\ell'}{2}\omega\|\geq \|q_m\omega\|.
\end{align}
This implies, similar to \eqref{eq:Sw_Dio}, \eqref{eq:Sw_Dio_21} and \eqref{eq:Sw_Dio_22} that
\begin{align}
    \tilde{S}_w\geq 2\min_{r=1}^{4s}\min_{\ell\in T_r}\ln \|2\theta-1/2+(\frac{x_1+\ell}{2}+4sq_m-3)\omega\|-(2\beta+4)\varepsilon sq_m>-(2\beta+45)\varepsilon sq_m.
\end{align}
This proves the claimed result in Lemma \ref{lem:allmin_22_Dio}.
\end{proof}
Combining Lemma \ref{lem:allmin_22_Dio} with \eqref{eq:sum22_Dio}, we have
\begin{align}\label{eq:sum22_nonres_Dio}
\sum_{2,2}\geq -4sq_m\ln 2-(8\beta+200)\varepsilon sq_m.
\end{align}
This proves Lemma \ref{lem:sum22_Dio}.
\end{proof}

Combining \eqref{eq:sum1_nonres_Dio}, \eqref{eq:sum2_nonres_Dio} with Lemmas \ref{lem:sum21_Dio} and \ref{lem:sum22_Dio}, we arrive at
\begin{align}
\sum_1-\sum_2 &\leq 4sq_m(-\ln2+\varepsilon)-4sq_m\ln 2-\sum_{2,1}-\sum_{2,2}\\
&\leq (16\beta+1000)\varepsilon sq_m\leq \frac{\varepsilon_0 h}{2}.
\end{align}
This proves Lemma \ref{lem:non_res_uni_Dio}.
\end{proof}


\bibliographystyle{amsplain}

\end{document}